\documentclass[reqno, final, noamsfonts, e-only]{amsart}
\usepackage{amsthm, mathtools, mleftright, lmodern, mathabx, xcolor, bm}
\usepackage{ntrees}
\usepackage{array}
\usepackage{bookmark, cite}
\usepackage[nameinlink]{cleveref}
\hypersetup{%
  citecolor = blue,
  linktocpage = true,
  pdftitle = {Geometry of controlled rough paths},
  pdfauthor = {Mazyar Ghani, Sebastian Riedel, Alexander Schmeding, Nikolas Tapia},
  bookmarksopen = true,
  bookmarksopenlevel = 1,
  unicode = true,
}
\usepackage[ocgcolorlinks]{ocgx2}
\usepackage[margin={3cm,2cm}]{geometry}
\usepackage[bb=esstix, scr=boondox, cal=esstix, frak=esstix, scaled=1]{mathalfa}
\makeatletter
\newcommand\@dotsep{5}
\makeatother

\DeclareMathOperator{\Prim}{Prim}

\def\id{\mathrm{id}}

\def\btop{\mathbin{\top}}
\DeclarePairedDelimiter{\cnorm}{\vvvert}{\vvvert}
\DeclarePairedDelimiter{\fnorm}{[\mkern-4mu]}{[\mkern-4mu]}

\newtheorem{theorem}{Theorem}[section]
\newtheorem*{theorem*}{Theorem}

\theoremstyle{plain}

\newtheorem{corollary}[theorem]{Corollary}

\newtheorem{lemma}[theorem]{Lemma}
\newtheorem{prop}[theorem]{Proposition}

\numberwithin{equation}{section}

\theoremstyle{definition}
\newtheorem{defn}[theorem]{Definition}
\newtheorem{expl}[theorem]{Example}
\newtheorem{numba}[theorem]{}

\newenvironment{example}{\pushQED{\qed}\begin{expl}}{\popQED\end{expl}}
\crefname{numba}{paragraph}{paragraphs}
\Crefname{numba}{Paragraph}{Paragraphs}

\theoremstyle{remark}
\newtheorem{remark}[theorem]{Remark}

\newcommand{\BRP}{\ensuremath{\mathscr{C}}}
\newcommand{\R}{\mathbb{R}}
\newcommand{\Q}{\mathbb{Q}}
\newcommand{\N}{\mathbb{N}}

\newcommand{\bX}{\mathbf{X}}

\newcommand{\bY}{\mathbf{Y}}
\newcommand{\bZ}{\mathbf{Z}}
\newcommand{\cH}{\mathcal{H}}
\def\cF{\mathcal{F}}
\def\cT{\mathcal{T}}
\def\cP{\mathcal{P}}
\newcommand{\D}[2]{\mathscr{D}^{#1}_{#2}}
\def\FnSp{C^{0,1-\varepsilon}}
\def\Lip{\mathrm{Lip}}

\def\tdot{\Forest{[]}}
\def\tlI{\Forest{[[]]}}
\def\tlII{\Forest{[[[]]]}}
\def\tv{\Forest{[[][]]}}

\newcommand{\vertiii}[1]{{\left\vert\kern-0.25ex\left\vert\kern-0.25ex\left\vert #1
\right\vert\kern-0.25ex\right\vert\kern-0.25ex\right\vert}}

\newcommand{\g}{\mathfrak{g}}
\newcommand{\InfChar}[2]{ \g(#1, #2) }       % infinitesimal characters of the Hopf algebra #1 with values in the algebra #2
\newcommand{\Char}[2]{     G(#1, #2) }     % characters of the Hopf algebra #1 with values in the algebra #2
\newcommand{\LB}[1][\cdot \hspace{1pt} , \cdot]{\left[\hspace{1pt} #1 \hspace{1pt} \right]}

\newcommand{\dd}{\,\mathrm{d}}

\usepackage{etoolbox}
\mleftright
\begin{document}
\title{The geometry of controlled rough paths}
\author[M. Ghani]{Mazyar Ghani Varzaneh}
\address{Mathematics Institute, TU Berlin, Str. des 17. Juni 136, 10586 Berlin, Germany}
\email{mazyarghani69@gmail.com}
\author[S. Riedel]{Sebastian Riedel}
\address{Leibniz Universit\"at Hannover, Welfengarten 1, 30167 Hannover, Germany}
\email{riedel@math.uni-hannover.de}
\author[A. Schmeding]{Alexander Schmeding}
\address{Nord universitet, Høgskoleveien 27, 7601 Levanger, Norway}
\email{alexander.schmeding@nord.no}
\author[N. Tapia]{Nikolas Tapia}
\address{Weierstrass Institute, Mohrenstr. 39, 10117 Berlin, Germany \& Mathematics Institute, TU Berlin,
Str. des 17. Juni 136, 10586 Berlin, Germany}
\email{tapia@wias-berlin.de}
\urladdr{https://www.wias-berlin.de/people/tapia}
\thanks{SR acknowledges financial support by the DFG via Research Unit FOR 2402. NT is supported by the DFG MATH$^+$ Excellence Cluster}
\begin{abstract}
We prove that the spaces of controlled (branched) rough paths of arbitrary order form a continuous field of Banach spaces. This structure has many similarities to an (infinite-dimensional) vector bundle and allows to define a topology on the total space, the collection of all controlled path spaces, which turns out to be Polish in the geometric case. The construction is intrinsic and based on a new approximation result for controlled rough paths. This framework turns well-known maps such as the rough integration map and the It\^o-Lyons map into continuous (structure preserving) mappings. Moreover, it is compatible with previous constructions of interest in the stability theory for rough integration. 
\end{abstract}
\maketitle

\tableofcontents

\section{Introduction and statement of results}

One of the key insights in rough path theory is that there is no canonical integration theory that allows to integrate two arbitrary paths of low regularity against one another\footnote{This fact is visible e.g. in stochastic analysis: there is no exclusive notion of a stochastic integral, both It\^o and Stratonovich integral have their justification.}. Lyons' fundamental observation was that paths have to be augmented with higher order objects which play the role of iterated integrals in order to build a robust theory of controlled ordinary differential equations \cite{Lyo98}, and he called these augmented paths \emph{rough paths}. Later, Gubinelli realized that given a reference rough path, there is a canonical notion of an integral that allows to integrate paths that ``look like the reference path'' on small time scales, known as \emph{controlled rough paths} \cite{Gub04, Gub10}. These principles were carried over from the world of paths to the world of distributions (or generalized functions) by Hairer \cite{Hai14} and Gubinelli-Imkeller-Perkowski \cite{GIP15}. In Hairer's theory of regularity structures, the reference distribution has to be augmented with products of itself and its derivatives, and multiplying distributions is explained when they are \emph{modelled} after the reference distribution. 

Although these concepts are nowadays used extensively in the theory of stochastic ordinary and partial differential equations, the mathematical structure which is formed by rough paths and their controlled paths is still not very well understood. Since for every reference \(\alpha\)-rough path $\mathbf{X}$, the set of all controlled paths $\mathscr{D}^\alpha_{\mathbf{X}}$ constitutes a linear (Banach) space, it is natural to suspect that they form some sort of ``infinite-dimensional vector bundle'' \cite[p. 123]{Ina19}. However, making these ideas precise, one encounters several difficulties, starting with the fact that the space of rough paths does not carry any known (infinite-dimensional) manifold structure. Surprisingly, it turns out that there is still a non-canonical homeomorphism transforming the collection of spaces of controlled paths to a trivial infinite-dimensional vector bundle \cite[Remark 4.8]{FaH20}. However, this map is highly non-explicit since it uses the Lyons-Victoir extension theorem \cite{LaV07}, and the applicability of this result in practice is unclear. Furthermore, note that the existence of this homeomorphism does not imply that the spaces of controlled rough paths form a smooth (Banach) vector bundle.

One motivation for this article is an observation two of us made in \cite{GRS22}. In this work, we considered the solution map induced by linear stochastic delay differential equations (SDDE) driven by a Brownian motion, which turns out to be a linear map between spaces of controlled paths. Since these spaces were random, we faced serious measurability issues when considering, for example, the operator norm for this map. To overcome these issues, we proved that the spaces of controlled paths form a \emph{measurable field of Banach spaces} \cite[Definition 3.13]{GRS22}. It turns out that this structure is extremely useful and a perfect infinite-dimensional substitute for what is called a \emph{measurable bundle} \cite[1.9.2 Definition]{A1998}. In fact, we could prove that SDDEs induce random dynamical systems (RDS) on this measurable field of Banach spaces, and that important theorems like the multiplicative ergodic theorem hold for these RDSs \cite{GVR21}. Note, however, that we proved the existence of a measurable field only for Brownian-type rough paths, i.e. in the regime of an H\"older-index close to $\frac{1}{2}$. One goal of the present article is to generalize this important result to rough paths of arbitrary regularity.\footnote{We note that extending these results to arbitrary regularity is of interest, for example, for financial applications in the context of rough volatility models \cite{BFG2016,GJR2018,MW2021} where the typical regularity of the driving noise is of the order of \(\alpha=\tfrac1{10}\).}

In the present paper, we first construct a novel approximation of controlled rough paths by rough integrals of smooth functions and
smooth remainders. We do this for paths which are controlled by branched rough paths of arbitrary order.
It will turn out that once the algebra underlying the process has been worked out, it is straightforward to construct a
dense subset of controlled rough paths depending continuously on the underlying controlling path.
Our results here are a far reaching generalisation of our earlier results in \cite{GRS22}.
Based on this novel approximation result, we are able to construct a finer structure for the spaces of controlled rough
paths.
Indeed we prove that the spaces of controlled rough paths form a so-called \emph{continuous field of Banach spaces} (see \Cref{defn_cont:field}). This structure is well-known from the representation theory of $C^\ast$-algebras (see e.g.\ \cite{Dix77,FaD88}).
It will turn out that our field of Banach spaces sits in the middle between the two trivial bundles, namely for $\alpha
\in (1/3,1/2)$ we have
  $$\mathscr{C}^{\alpha} \times (C^\infty\times C^\infty) \subseteq \mathscr{C}^{\alpha} \ltimes \mathscr{D}^{\alpha} \subseteq \mathscr{C}^{\alpha} \times ({C}^{2\alpha} \times {C}^{\alpha})$$
where the trivial bundle on the right is the one from \cite{FaH20} and the trivial bundle on the left forms a ``dense''
subset of the field of Banach spaces (we will say more about the topology we use here below).

In more detail, our results subsume the following theorem:
\begin{theorem}
  Fix \(\alpha\in(0,1)\) and let \(\BRP^\alpha\) denote the space of $\alpha$-branched rough paths.
For \(\bX\in\BRP^\alpha\), denote by \(\D{\alpha}{\bX}\) the associated space of controlled paths.
Then the family \((\D{\alpha}{\bX})_{\bX\in\BRP^\alpha}\) forms a continuous field of Banach spaces (FoBS) over \(\BRP^\alpha\).
  \label{thm:main.intro}
\end{theorem}

\Cref{thm:main.intro} has several implications from which we will list a few here:

\begin{itemize}
  \item Every continuous field of Banach spaces allows to define an intrinsic topology (that we call the \emph{tube topology}) on the total space, i.e., in our case on
    the product space
\[
  \prod_{\bX\in\BRP^\alpha}\D\alpha\bX.
\]
It turns out (\Cref{sec:control-bundle}) that this topology is compatible with rough integration, therefore making the It\^{o}-Lyons map a \emph{continuous section} of the bundle. This also gives a nice interpretation of the ``metric'' that compares two controlled paths that live in different fibres (see \cite[p. 74]{FaH20}): adding the distance of the reference rough paths to it, we can show that this metric is exactly the one inducing the topology on the total space. This also implies that the total space is a Polish space when restricting to geometric rough paths\footnote{See \cite[Section 2.4]{FaH20} for a precise definition of geometric rough path.} (\Cref{thm:total_space_polish}).

\item The continuous field of Banach spaces induces a so-called \emph{Banach bundle} of controlled paths, cf. \Cref{sec:fobs}. It turns out that well-known mappings in rough paths theory such as the rough integration map and the It\^o-Lyons map are structure-preserving mappings on this Banach bundle. We would like to mention here that for proving these statement, we had to establish sharp bounds for the rough integral (\Cref{prp:int.bounded}) and local Lipschitz continuity of the It\^o-Lyons map (\Cref{thm:ItoLyonscts}) \emph{in full generality}, i.e. for rough integrals and rough differential equations driven by branched rough paths of any order. Although these bounds are widely accepted in the rough paths community, we could not find them explicitly worked out anywhere in the literature. Having closed this gap is another contribution of the present work.

    \item  In \cite[Remark 4.9]{FaH20}, it is claimed that ``the notion of 'controlled rough paths' (...) does not come with a natural approximation theory''. As a remedy, one can construct non-canonical approximations using the Lyons-Victoir extension theorem \cite[Exercise 4.8]{FaH20}, but this is not trivial as the reader can easily convince herself. Furthermore, the construction in the cited reference only works in the case $\tfrac13<\alpha<\tfrac12$. We will show that \Cref{thm:main.intro} immediately yield \emph{canonical} approximations for controlled rough paths of any order (\Cref{remark:smooth_approx}). In principle, this approximation result should yield, as particular cases, results concerning convergence and stability properties of random dynamical systems \cite{GRS22}, rough differential equations \cite{FaZ18,FaH20}, and numerical schemes for stochastic (partial) differential equations \cite{BaBaRaRaS20,LT2019}, but we do not explore this further and content ourselves with setting up the general framework underlying these results.

    \item Continuous fields of Banach spaces are also measurable fields. The results in this article allow to study dynamical properties of SDDE as in \cite{GRS22} driven by more general processes than Brownian motion, e.g. fractional Brownian motion with lower Hurst parameter.

\end{itemize}

The remainder of the article is organized as follows: in \Cref{sec:brp} we recall the definition of branched rough paths
and set up basic notation and results necessary for the next sections. In \Cref{sec:main} we introduce the main
construction of the paper, namely how to build a controlled rough path out of a collection of smooth functions.
Next, in \Cref{sec:short-time} we use this construction to show how to locally approximate any given controlled path by
piecewise-linear controlled paths in the Banach norm, and then extend this approximation to the full time interval in
\Cref{sec:local-global}. In \Cref{sec:fobs} we recall the main definition of a continuous field of Banach spaces and
show that the bundle of controlled paths satisfies this definition with the space of branched rough paths as base space,
and use this result to construct the so-called tube topology.
We end by showing several continuity results under this particular topology for well-known objects associated to rough paths, such as the
It\^{o}-Lyons map, in \Cref{sec:control-bundle}.

\section{Branched rough paths}\label{sec:brp}
Let us first recall the Hopf algebraic framework for branched and controlled rough paths.
In what follows, we write $\N = \{1,2,3,\ldots\}$ for the set of natural numbers and $\N_0 = \N \cup \{0\}$. All linear structures are defined over the field of real numbers, denoted by \(\R\).

\begin{numba}
  Let \(A\) be a finite, non-empty set, and denote by \(\mathcal T\) the linear span of all rooted trees decorated with labels from $A$.
  We recall that the \emph{decorated Connes--Kreimer Hopf algebra} \(\cH\) is the polynomial algebra
  $\cH = S(\mathcal{T})$, which can be identified with the vector space spanned by all forests of
  decorated trees; the basis will be denoted by \(\cF\).
  We grade $\cH$ by the number of nodes, and we denote the \emph{degree of \(h\in\cH\)} by \(|h|\).
  There is a unique forest of degree zero, called the \emph{empty forest} and denoted by \(\bm 1\).
  We also set, for each \(n\ge 0\),
  \[
    \cF_n\coloneq\{h\in\cF:|h|=n\},\quad \cF_{(n)}\coloneq\{h\in\cF:|h|\le n\},\quad\cF_{(n)}^<\coloneq\cF_{(n)}\setminus\cF_n
  \]
  and we let \(\cH_n\) and \(\cH_{(n)}\) denote the corresponding linear spans of $\cF_n$ and $\cF_{(n)}$, respectively.
  We finally introduce the set \(\cF^+\coloneq\cF\setminus\{\bm1\}\) and its linear span
  \[
    \cH^+=\bigoplus_{n>0}\cH_n.
  \]

The space \(\cH\) comes with a collection of maps \(([\cdot]_a, a\in A)\),
where \([\tau_1\dotsm\tau_n]_a\) is obtained by grafting the trees
  \(\tau_1,\dotsc,\tau_n\) to a new root labeled by \(a\).
  Observe that by definition, \(|[h]_a|=|h|+1\).
  Moreover, for each tree \(\tau\in\mathcal T\) with \(|\tau|=n+1\) there is a unique label
  \(a\in A\) and a unique forest \(h\) such that \([h]_a=\tau\).
  These maps define uniquely a \emph{coproduct} \(\Delta\colon\cH\to\cH\otimes\cH\) via the
  relations \(\Delta\bm 1=\bm 1\otimes\bm 1\), \(\Delta(h_1\dotsm
  h_n)=\Delta h_1\dotsm\Delta h_n\), and
  \begin{equation}
  \label{eq:delta.tree}
    \Delta [h]_a=({\operatorname{id}}\otimes [\cdot]_a)\Delta h+[h]_a\otimes\bm 1
  \end{equation}
  for all \(a\in A\) and \(h, h_1,\dotsc,h_n\in\cH\).
  One can immediately check from this definition that \(\Delta\) is coassociative, i.e., the
  identity \((\Delta\otimes\id)\circ\Delta=(\id\otimes\Delta)\circ\Delta\) holds.
  Therefore, the triple \((\cH,\cdot,\Delta)\) defines a bialgebra.

  It can be shown that for trees \(\tau\in\mathcal T\), the image \(\Delta\tau\) admits a representation
  in terms of \emph{admissible cuts} of \(\tau\) \cite{CK1998}, but we refrain from providing
  further details since we will not make use of it in the following.
  We shall instead use sumless Sweedler's notation
  \[
    \Delta h=h_{(1)}\otimes h_{(2)},\quad\Delta'h=h'\otimes h''
  \]
  where $\Delta'h\coloneq\Delta h-h\otimes\bm 1-\bm 1\otimes h$ denotes the \emph{reduced coproduct}.
  Since product and coproduct respect the grading we obtain a graded and connected bialgebra of
  finite-type, i.e.\ disassembling degreewise we obtain $\cH=\bigoplus_{n \in \N_0} \cH_n$ where
  $\cH_0=\R\bm1$ and $\dim \cH_n < \infty$. Note that any graded connected bialgebra is automatically a
  Hopf algebra whose antipode $S$ can be computed recursively (whence we do not discuss it here, but see \cite{manHopf}).

  We introduce the \emph{n-fold iterated coproduct}
  \(\Delta^{(n)}\colon\cH\to\cH^{\otimes(n+1)}\) inductively by \(\Delta^{(0)}=\mathrm{id}\), and
  \(\Delta^{(n)}=(\Delta^{(n-1)}\otimes\mathrm{id})\circ\Delta\) for all \(n\ge 1\).
  Note that these are indeed well defined by coassociativity of \(\Delta\).
  An iterated reduced coproduct can be defined in a similar way.
  In particular, we can write the image of an element \(h\in\cH\) using Sweedler's notation as
  \[
    \Delta^{(n)}h=h_{(1)}\otimes\dotsm\otimes h_{(n+1)},\quad(\Delta')^{(n)}h=h^{(1)}\otimes\dotsm\otimes
    h^{(n+1)}.
  \]
  Finally, we remark that the \(n\)-fold tensor product \(\cH^{\otimes n}\) is also graded, with
  graded components
  \[
    \left( \cH^{\otimes n}
    \right)_k=\bigoplus_{j_1+\dotsb+j_n=k}\cH_{j_1}\otimes\dotsm\otimes\cH_{j_n},
  \]
  and so, since \(\Delta\) is a graded map, we see that if \(|h|=k\) then
  \(|h_1|+\dots+|h_{n+1}|=k\) for all \(n\ge 0\).
\end{numba}

  \begin{numba}
    Denote by $\cH^*$ the dual space of $\cH$, i.e.\ the space of all linear mappings $\cH \rightarrow \R$ and let $m_{\mathbb R}$ be the multiplication map of real numbers. We write $\langle \phi , h\rangle$ for the duality pairing between $\cH^*$ and $\cH$ given by evaluation. The convolution product
    \[ \phi \star \psi \coloneq m_{\mathbb R} \circ (\phi \otimes \psi ) \circ \Delta \]
    turns $\cH^*$ into a unital algebra (the unit being the counit $\varepsilon$ of $\cH$).
    Then the \emph{character group} $(\Char{\cH}{\R}, \star)$ of $\cH$ is defined as
    \begin{displaymath}
      \Char{\cH}{\R} \coloneq \left\{\phi \in \cH^* : \langle \phi ,ab\rangle = \langle \phi ,a\rangle\langle \phi ,b\rangle, \forall a,b \in \cH \text{ and } \langle \phi ,1_{\cH}\rangle = 1\right\}.
    \end{displaymath}
    Inversion in $\Char{\cH}{\R}$ is induced by precomposition with the antipode, i.e.\ for a character $\iota (\varphi) =\varphi^{-1} = \varphi \circ S$. The counit $\varepsilon$ of $\cH$ thus becomes the unit of the character group. One can show,\cite{BaDaS16} that the character group carries a natural Lie group structure whose Lie algebra is given by the \emph{Lie algebra of infinitesimal characters} $(\InfChar{\cH}{\R},\LB[\cdot,\cdot])$;
    $$\InfChar{\cH}{\R} \coloneq \{\psi \in \text{Hom} (\cH,\R) : \langle \psi, ab\rangle = \langle \psi, a\rangle\langle\varepsilon, b\rangle + \langle\varepsilon,a\rangle\langle \psi , b\rangle, \forall a,b\in \cH\}$$
    with the Lie bracket $\LB[\psi, \kappa] \coloneq \psi \star \kappa - \kappa \star \psi$.

    There is an injection \(\cH\hookrightarrow\cH^*\) mapping \(h\in\cH\) to the linear functional \(h'\mapsto\delta_{h,h'}\) (the Kronecker delta mapping $h$ to $1$).
    In the sequel we will just identify \(h\in\cH\) with this functional.
    Using this identification, we write dual elements \(X\in\cH^*\) as formal forest series, i.e.,
    \[
      X=\sum_{h\in\cF}X^h h
    \]
    where \(X^h=\langle X,h\rangle\in\R\).
  \end{numba}

\begin{numba}
  \label{par:prim}
  An important subspace is the space of \emph{primitive elements}
  \[
    \Prim\coloneq\{h\in\cH:\Delta'h=0\}.
  \]
  Let \(h\) and \(h'\) be forests.
  We define the \emph{natural growth operation} \(h\btop h'\) as the sum of all forests obtained by
  grafting \(h\) to every node of \(h'\), normalized by \(|h'|\).
  This definition is extended bilinearly to the natural growth operator \(\btop\colon\cH\otimes\cH\to\cH\).
  For example
  \begin{align*}
    \tdot\tdot\btop\tlI&= \frac12\left( \Forest{[[][][]]}+\Forest{[[[][]]]}
    \right)\\
    \tlI\btop\tdot\tdot&= \frac12\left(
    \tlII\tdot+\tdot\tlII \right)=\tlII\tdot
  \end{align*}
  We observe that this operator is neither associative nor commutative.
  Given a collection \(h_1,\dotsc,h_n\in\cH\), we set
  \[
    \top(h_1,\dotsc,h_n)=(\dotsb((h_1\btop h_2)\btop h_3)\dotsb)\btop h_n.
  \]

  The following results are due to Foissy:
  \begin{lemma}{{\cite[Lemma 4.3 and Theorem 9.6]{F2002}}}\label{Foissy_Basis}
    \label{lem:foissy}
  \begin{enumerate}
    \item If \(p_1,\dotsc,p_n\in\mathrm{Prim}\), then
    \[
      \Delta'\top(p_1,\dotsc,p_n)=\sum_{j=1}^{n-1}\top(p_1,\dotsc,p_j)\otimes\top(p_{j+1},\dotsc,p_n).
    \]
    \label{lem:ftopc}
  \item Let \(\cP=\{p_i:i\ge 1\}\) be a basis for \(\Prim\). Then
    \[
      \cP^\top\coloneq\{\top(p_{i_1},\dotsc,p_{i_k}):i_1,\dotsc,i_k\ge 1,k\ge 1\}
    \]
    is a basis for \(\cH\).
  \end{enumerate}
  \end{lemma}

  As before, for \(n\in\N\) we define
  \[
    \cP_{(n)}\coloneq\cP\cap\cH_{(N)},\quad\cP_n\coloneq\cP\cap\cH_n,\quad\cP_{(n)}^<\coloneq\cP_{(n)}\setminus\cP_n,
  \]
  and simialrly for \(\cP^\top_{(n)}\) and so on.

  As a consequence, we have that for every forest \(h\in\cF\), there are coefficients
  \((c_\rho(h):\rho\in\cP^\top)\) with only finitely many being non zero, such that
  \[
    h=\sum_{\rho\in\cP^\top}c_{\rho}(h)\rho.
  \]

  In particular, every element (where only finitely many coefficients \(Z^h\) are non-zero)
  \[
      Z = \sum_{h\in\cF}Z^h h\in\cH
  \]
  can be rewritten as
  \[
    Z = \sum_{\rho\in\cP^\top}Z^{\rho}\rho,
  \]
  with
  \[
    Z^\rho=\sum_{h\in\cF}c_\rho(h)Z^h.
  \]

  Denote by \(\{f_{i_1,\dotsc,i_k}\}\) the basis of the graded dual \(\cH^{\mathrm{gr}}\), dual to \(\mathcal P^\top\), i.e., such that
  \(\langle f_{i_1,\dotsc,i_k},\top(p_{j_1},\dotsc,p_{j_l})\rangle=1\) if and only if \(k=l\) and \(i_1=j_1,\dotsc,i_k=j_k\).
  We will also use the notation \(\rho^*\coloneq f_{i_1,\dotsc,i_k}\) whenever
  \(\rho=\top(p_{i_1},\dotsc,p_{i_k})\in\cP^\top\).
  In particular, we identify the change of basis coefficients as \(c_\rho(h)\coloneq\langle\rho^*,h\rangle\), and we
  have the identity
  \[
    \sum_{h\in\cF}c_\rho(h)h=\rho^*.
  \]

  Therefore, any dual forest series
  \[
    X=\sum_{h\in\cF}X^h h\in\cH^*
  \]
  can be be re-expanded in the new basis as
  \[
    X=\sum_{\rho\in\cP^\top}X^\rho\rho^*
  \]
  where
  \[
    X^\rho=\sum_{h\in\cF}c_\rho(h)X^h.
  \]
\end{numba}

\begin{numba}
  Given a rooted tree \(\tau\) and a vertex \(v\in V(\tau)\), let \(\tau_v\) denote the subtree of
  \(\tau\) with \(v\) as root.
  For \(v\in V(\tau)\), let \(SG(\tau,v)\) be the group of permutations of identical branches out of
  \(v\), i.e., if \(\{v_1,\dotsc,v_k\}\) are the children of \(v\), then \(SG(\tau,v)\) is the group
  generated by the permutations that exchange \(\tau_{v_i}\) and \(\tau_{v_j}\) when they are
  isomorphic rooted trees.
  The \emph{symmetry group of \(\tau\)} is the direct product
  \[
    SG(\tau)\coloneq\prod_{v\in V(\tau)}SG(\tau,v),
  \]
  and the \emph{symmetry factor \(\Sigma(\tau)\)} of \(\tau\) is defined to be the order of
  \(SG(\tau)\) \cite{H2003}.
  For a forest \(h\in\cH\), we let \(\Sigma(h)\coloneq\Sigma(\mathfrak I_j(h))\) for some
  \(j\in A\).
  It is not hard to see that in fact this definition is independent of the choice of \(j\in A\).

  Letting \(\zeta_h(h')=\Sigma(h)\delta_{h,h'}\), one can show that \(\{\zeta_h:h\in\cF\}\) is a
  basis of \(\cH^*\) dual to the forest basis \cite[Proposition 4.4]{H2003}, and in particular
  \(\zeta_{\bm1}=\varepsilon\).
\end{numba}

  \begin{numba}\label{numba:pre-Lie}
    Recall that a (right) pre-Lie algebra is a vector space \(V\) with a bilinear operator \(\triangleleft\colon V\otimes
    V\to V\) such that the associator \(\mathrm a_\triangleleft(x,y,z)\coloneq(x\triangleleft y)\triangleleft
    z-x\triangleleft(y\triangleleft z)\) is symmetric in the last two variables, i.e.,
    \[
      \mathrm a_\triangleleft(x,y,z)=\mathrm a_\triangleleft(x,z,y)
    \]
    for all \(x,y,z\in V\).

    There is a left pre-Lie structure on \(\cT\) given by grafting of trees, denoted by
    \(\curvearrowleft\colon\cT\otimes\cT\to\cT\).
    \begin{example}
      \begin{align*}
        \tdot\curvearrowleft\tv&= \Forest{[[[][]]]}\\
        \tv\curvearrowleft\tdot&= 2\Forest{[[[]][]]}+\Forest{[[][][]]}
      \end{align*}
    \end{example}
    This is in fact the free pre-Lie algebra on \(d\) generators \cite{CL2001}.

    Grafting can be extended to an operator \(\curvearrowleft\colon\cT\otimes\cH\to\cT\) by grafting
    every forest on the right to some node of the tree on the left.
    \begin{example}
      \begin{align*}
        \tdot\curvearrowleft\tdot\tdot&= \tv,\\
        \tlI\curvearrowleft\tdot\tdot&= \Forest{[[[][]]]}+2\Forest{[[[]][]]}+\Forest{[[][][]]}
      \end{align*}
    \end{example}

    Let \(V\) be a vector space. Let us recall that the symmetric algebra \(S(V)\) carries a coproduct \(\Delta_*\)
    defined by \(\Delta_*v=v\otimes 1+1\otimes v\) for all \(v\in V\).
    We stick to the sumless Sweedler's notation for this coproduct as well.
    A \emph{symmetric brace algebra} \cite{LM2005} is a vector space \(V\) equipped with a brace \(V\otimes S(V)\to V, x\otimes a\mapsto
    x\{a\}\) such that
    \begin{align*}
      x\{1\}&= x\\
      x\{y_1\dotsm y_n\}\{a\}&= x\{y_1\{a_{(1)}\}\dotsm y_n\{a_{(n)}\}a_{(n+1)}\}.
    \end{align*}

    It can be shown that grafting endows \(\cT\) with the structure of a symmetric brace~\cite{OG2008}.\footnote{In general, this holds
      for any pre-Lie algebra, i.e., if \((V,\triangleleft)\) is pre-Lie, then \( x\{a\}=x\triangleleft a\) is a symmetric
    brace for some suitable extension of \(\triangleleft\) to \(S(V)\) on the right.}
    In particular, the identity
    \begin{equation}
      (\tau\curvearrowleft\bar{h}_1)\curvearrowleft\bar{h}_2=\tau\curvearrowleft(\bar{h}_1\star\bar{h}_2)
      \label{eq:arrow.star}
    \end{equation}
    holds for any \(\tau\in\cT\) and \(\bar{h}_1,\bar{h}_2\in\cF\).

    Finally, we extend \(\curvearrowleft\) to \(\cH\otimes\cH\) via
    \[
      h_1h_2\curvearrowleft\bar h=(h_1\curvearrowleft\bar h_{(1)})(h_2\curvearrowleft\bar h_{(2)}).
    \]
    \begin{example}
			Since \(\tdot\) is primitive and
      \[
        \Delta_*\tdot\tdot=\tdot\tdot\otimes\bm 1+2\,\tdot\otimes\tdot+\bm 1\otimes\tdot\tdot
      \]
      we have
      \begin{align*}
        \tdot\tdot\curvearrowleft\tdot&= 2\tdot\tlI,\\
        \tdot\tdot\curvearrowleft\tdot\tdot&= 2\tv\tdot+2\tlI\tlI.\qedhere
      \end{align*}
    \end{example}
    Furthermore, the \(\star\) product admits the following description:
    \[
      \zeta_{\bar h}\star \zeta_h=\zeta_{(h\curvearrowleft \bar h_{(1)})\bar h_{(2)}}.
    \]
    More concretely, we have the formula
    \[
      \zeta_{\bar h}\star \zeta_{\tau_1\dotsm\tau_n}=\zeta_{(\tau_1\curvearrowleft\bar h_{(1)})\dotsm(\tau_n\curvearrowleft\bar h_{(n)})\bar h_{(n+1)}}.
    \]
    \begin{example}
      We compute all products of forests up to degree 3:
      \begin{align*}
        \zeta_{\tdot}\star\zeta_{\tdot}&=
        \zeta_{\tdot\tdot}+\zeta_{\tlI}&&\\
        \zeta_{\tdot\tdot}\star\zeta_{\tdot}&=
        \zeta_{\tdot\tdot\tdot}+2\zeta_{\tdot\tlI}+\zeta_{\tv}&
        \zeta_{\tdot}\star\zeta_{\tdot\tdot}&=
        \zeta_{\tdot\tdot\tdot}+2\zeta_{\tdot\tlI}\\
        \zeta_{\tlI}\star\zeta_{\tdot}&=\zeta_{\tdot\tlI}+\zeta_{\tlII}
        &\zeta_{\tdot}\star\zeta_{\tlI}&=
        \zeta_{\tdot\tlI}+\zeta_{\tlII}+\zeta_{\tv}.
      \end{align*}
      In particular we see that the \(\star\) product is not commutative.
      \label{exp:prod}
    \end{example}
		\begin{remark}
			When written in terms of the ``pure basis'' \(\cF\), there are some non-trivial factors in front of each term in the \(\star\) product. For
			example
			\[
				\tdot\star\tdot=2\,\tdot\tdot+\tlI,\quad\tdot\tdot\star\tdot=3\,\tdot\tdot\tdot+\tdot\tlI+\tv.
			\]
			This becomes even more apparent in the case of decorated trees, as
			\[
				\Forest{[i]}\star\Forest{[j]}=(1+\delta_{i,j})\Forest{[i]}\Forest{[j]}+\Forest{[j[i]]}.
			\]
			The formula in terms of the \(\zeta_h\) basis stays, however, the same in both cases.
		\end{remark}

    \begin{lemma}
      \label{lem:brace}
      Let \(h,\bar h\in\cH\) and denote by \(\Pi\colon \cF\to\cT\) the projection onto trees.
      The identity
      \[
        \Pi\left( \bar h\star [h]_a \right)=[\bar h\star h]_a
      \]
      holds.
    \end{lemma}
    \begin{proof}
      Let us start by noticing that \( [h]_a=\Forest{[a]}\curvearrowleft h. \)
      Now, from the formula for the product in terms of grafting we see that
      \begin{align*}
        \zeta_{\bar h\star [h]_a}&= \zeta_{([h]_a\curvearrowleft\bar h_{(1)})\bar h_{(2)}}\\
        &= \zeta_{\left( (\Forest{[a]}\curvearrowleft h)\curvearrowleft\bar h_{(1)} \right)\bar h_{(2)}}.
      \end{align*}
      Therefore, by \cref{eq:arrow.star},
      \begin{align*}
        \zeta_{\Pi\left( \bar h\star [h]_a \right)}&= \zeta_{(\Forest{[a]}\curvearrowleft h)\curvearrowleft\bar h}\\
        &= \zeta_{\Forest{[a]}\curvearrowleft(\bar h\star h)}\\
        &= \zeta_{[\bar h\star h]_a}.
      \end{align*}
      The proof is finished by noting that the map \([\cdot]_a\) leaves the symmetry factor invariant.
  \end{proof}
\end{numba}

  \begin{numba}
    Let \(n \in \N\), then we denote by \(C_n\) the vector space of continuous scalar functions on
    \([0,1]^n\), vanishing whenever two contiguous arguments coincide.
    More precisely, \(C_n\) consists of functions \(f\coloneq[0,1]^n\to\R\) such that \(f_{t_1\dotsm
    t_n}=0\) when \(t_i=t_{i+1}\) for some index \(i\in\{1,\dotsc,n-1\}\).
    As a convention, we set \(C_0=\R\).

    Given \(f\in C_n\), we define \(\delta f\in C_{n+1}\) by setting
    \[
      \delta f_{t_1\dotsm t_{n+1}}=\sum_{k=1}^{n+1}(-1)^{k}f_{t_1\dotsm\hat t_{k}\dotsm t_{n+1}}
    \]
    where \(\hat t_k\) means that this argument is omitted.
    For instance, if \(f\in C_1\) then \(\delta f_{s,t}=f_t-f_s\); and if \(f\in C_2\) then \(\delta
    f_{s,u,t}=f_{s,t}-f_{s,u}-f_{u,t}\) and so on.
    It is a known fact that if \(f\in C_2\) is such that \(\delta f=0\) then \(f=\delta g\) for some
    \(g\in C_1\).

    Given \(f\in C_2\) and \(\alpha>0\), we define
    \[
      \|f\|_\alpha\coloneq\sup_{s\neq t}\frac{|f_{s,t}|}{|t-s|^\alpha}
    \]
    and we set \(C_2^\alpha\coloneq\{f\in C_2:\|f\|_\alpha<\infty\}\).
  \end{numba}

  Before we continue let us fix some useful notation (which has the unfortunate sideffect of identifying the dual and primal Hopf algebra structures).
  \begin{numba}
    Recall that $\cH \rightarrow \cH^\ast$ is an injection defined on $h \in \mathcal{F}$ by the linear functional $\delta_h$ (and we suppress the identification in the notation), hence it makes sense to use both the evaluation \(\bX^h_{s,t}\coloneq\langle\bX_{s,t},h\rangle\) and $X_{s,t} \star h \coloneq X_{s,t} \star \delta_h$.
  \end{numba}

  \begin{numba}
    \label{numba:brp}
    For $\alpha \in (0,1)$, an \emph{\(\alpha\)-Hölder branched rough path} is a family of characters \((\bX_{s,t}:s,t\in[0,T])\) over \(\cH\)
    such that
    \[
      |\langle\bX_{s,t},h\rangle|\lesssim|t-s|^{\alpha|h|}
    \]
    for all \(h\in\cH\), and \(\bX_{s,u}\star\bX_{u,t}=\bX_{s,t}\) for all \(s,u,t\in[0,T]\).

    With the help of the notation introduced in the previous paragraph, we remark that Chen's identity
    can be rewritten as\begin{equation}
      \delta\bX_{s,u,t}^h=\langle\bX_{s,u}\otimes\bX_{u,t},\Delta'h\rangle=\bX_{s,u}^{h'}\bX_{u,t}^{h''}.
      \label{eq:chencop}
    \end{equation}
    This identity together with \cref{eq:delta.tree} implies that for all labels \(a\in A\) we have
    \[
      \delta\bX_{s,u,t}^{[h]_a}=\bX_{s,u}^h\bX^{\Forest{[a]}}_{u,t}+\bX_{s,u}^{h'}\bX_{u,t}^{[h'']_a}.
    \]
    In particular, if \(h\in\Prim(\cH)\), there exists a path \(\Gamma^h\in C_1\) such that
    \(\bX^h_{s,t}=\delta\Gamma^h_{s,t}\).
    In the case where \(h=\Forest{[a]}\) for some \(a\in A\), we just write \(X^a\).
    We observe that this path is in general \textbf{not} unique, but there is a canonical choice with
    \(\Gamma^h_0=0\).
    We also note that by definition, \(\Gamma^h\) is \(\alpha|h|\)-Hölder continuous.

    We will denote by \(\BRP^\alpha\) the set of \(\alpha\)-Hölder branched rough paths.
    We also set \(N\coloneq\lfloor\alpha^{-1}\rfloor\) so that \(N\alpha\le 1<(N+1)\alpha\).

    We endow the set \(\BRP^\alpha\) with the distance
    \[
      \rho_\alpha(\bX,\tilde\bX)\coloneq\max_{h\in\cF^+_{(N)}}\|\bX^h-\tilde{\bX}^h\|_{|h|\alpha}
    \]
    and we define
    \[
      \cnorm{\bX}_\alpha\coloneq\rho_\alpha(\bm 1,\bX)=\max_{h\in\cF^+_{(N)}}\|\bX^h\|_{|h|\alpha}.
    \]
  \end{numba}

  \begin{numba}
    Let \(\bX\in\BRP^\alpha\) for $\alpha \neq 1/n$ for any $n \in \N$.\footnote{This technical restriction on $\alpha$ is standard, cf.\ \cite{LaV07} and we will require it from now on.}
    A \emph{path controlled by \(\bX\)} is a path \(\bZ\colon[0,T]\to\cH_{(N)}^<\) such that
    \[ |\langle h,\bZ_t\rangle-\langle \bX_{s,t}\star h,\bZ_s\rangle|\lesssim|t-s|^{(N-|h|)\alpha} \]
    for all \(h\in\cF_{(N)}^<\).
    We set \(R^h_{s,t}\coloneq\langle h,\bZ_t\rangle-\langle \bX_{s,t}\star h,\bZ_s\rangle\) and we
    note that this condition is equivalent to requiring that \(R^h\in C^{(N-|h|)\alpha}_2\) for all
    \(h\in\cF_{(N-1)}\).
    We denote the space of controlled paths by \(\mathscr D^{\alpha}_{\bX}\).
    It is a Banach space when endowed with the norm
    \[ \cnorm{\bZ}_\alpha\coloneq\sum_{h\in\cF_{(N)}^<}\left( |\langle h,\bZ_0\rangle|+\|R^h\|_{(N-|h|)\alpha} \right). \]
    Given \(\beta<\alpha\), we let \( \mathscr{D}_{\bX}^{\alpha,\beta}\coloneq\overline{\mathscr{D}_{\bX}^{\alpha}}^{\cnorm{\cdot}_\beta} \).
    
    If $\bZ \in \mathscr D^{\alpha}_{\bX}$ and $\tilde{\bZ} \in \mathscr D^{\alpha}_{\tilde{\bX}}$, we define
    \begin{align}\label{norm:controlled}
      \cnorm{\bZ ; \tilde{\bZ}}_{\alpha} \coloneq \sum_{h\in\cF_{(N)}^<}\left(|\langle h,\bZ_0-\tilde\bZ_0\rangle|+ \|R^h -
      \tilde{R}^h\|_{(N-|h|)\alpha}\right)
    \end{align}
    where \(\tilde{R}^h_{s,t}\coloneq\langle h,\tilde{\bZ}_t\rangle-\langle \tilde{\bX}_{s,t}\star h,\tilde{\bZ}_s\rangle\).

    Unraveling the definition of controlledness in coordinates, we see that the path
    \(Z_t^h\coloneq\langle h,\bZ_t\rangle\) has to satisfy for all $h\in\cF_{(N)}^<$ and the product from \Cref{numba:pre-Lie} the following
    relation to be controlled:
    \begin{equation}
      \label{eq:ctrlh}
      \delta Z_{s,t}^h=\sum_{\bar h\in\cF_{(N-|h|-1)}^+}\langle \bar h\star h,\mathbf{Z}_s\rangle\bX_{s,t}^{\bar h}+R^h_{s,t}
    \end{equation}
    This coincides with Gubinelli's definition in \cite{Gub10}.
    \begin{example}
      \label{ex:ctrl}
      Suppose \(\alpha\in(\tfrac15,\tfrac14)\). A path \(\bZ\in\mathscr D_{\bX}^{\alpha}\) satisfies
      (cf. \cite[Example 8.2]{Gub10}),
      \begin{align*}
        \delta Z_{s,t}^{\bm1}&=
        Z^{\tdot}_s\bX_{s,t}^{\tdot}+Z_s^{\tdot\tdot}\bX_{s,t}^{\tdot\tdot}+Z_{s}^{\tlI}\bX_{s,t}^{\tlI}+Z_s^{\tdot\tdot\tdot}\bX_{s,t}^{\tdot\tdot\tdot}+Z_s^{\tdot\tlI}\bX_{s,t}^{\tdot\tlI}
        +Z_{s}^{\tv}\bX_{s,t}^{\tv}+R^{\bm1}_{s,t},\\
        \delta Z_{s,t}^{\tdot}&=
        \left( 2Z_s^{\tdot\tdot}+Z_s^{\tlI} \right)\bX_{s,t}^{\tdot}+\left(
        3Z_s^{\tdot\tdot\tdot}+Z_s^{\tdot\tlI}+Z_s^{\tv} \right)\bX_{s,t}^{\tdot\tdot}+\left(
        Z_s^{\tdot\tlI}+Z_s^{\tlII} \right)\bX_{s,t}^{\tlI}+R_{s,t}^{\tdot},\\
        \delta Z_{s,t}^{\tdot\tdot}&= \left( 3Z_s^{\tdot\tdot\tdot}+Z_s^{\tdot\tlI}
        \right)\bX_{s,t}^{\tdot}+R_{s,t}^{\tdot\tdot},\\
        \delta Z_{s,t}^{\tlI}&= \left( Z_s^{\tdot\tlI}+Z_s^{\tlII}+Z_s^{\tv} \right)\bX_{s,t}^{\tdot}+R_{s,t}^{\tlI}
      \end{align*}
      with \(Z^{\tdot\tdot\tdot},Z^{\tlII},Z^{\tdot\tlI},Z^{\tv}\in C^{\alpha}\),
      \(R^{\tdot\tdot},R^{\tlI}\in C_2^{2\alpha}\), \(R^{\tdot}\in C_2^{3\alpha}\) and \(R^{\bm1}\in
      C_2^{4\alpha}\), where we have used the computations in \Cref{exp:prod}.
    \end{example}

    Thanks to \Cref{par:prim} the controlledness condition can be rewritten in term of the basis \(\cP^\top\) relative
    to a basis \(\cP\) of \(\Prim\). Indeed, for any \(h\in\cF_{(N)}^<\) the remainder rewrites as
    \[
      R^h_{s,t}=\delta Z^h_{s,t}-\sum_{\rho\in\cP^{\top,<}_{(N-|h|)}}Z^{\rho^*\star h}X^\rho_{s,t}.
    \]
  \end{numba}
  \begin{numba}\label{goodintegrands}
    Controlled paths are ``good integrands'' for rough paths, in the sense that if \(\bZ\in\mathscr
    D^{\alpha}_{\bX}\) then for any label \(a\in A\),
    \[
      \int_0^t Z_s\,\mathrm d\bX^a_s\coloneq\lim_{|\pi|\to
      0}\sum_{[u,v]\in\pi}\sum_{h\in\cF_{(N)}^<}Z_u^h\bX^{[h]_a}_{u,v}
    \]
    exists, and defines what is know as the \emph{rough integral of \(\bZ\) against \(X^a\)}.
    It satisfies the fundamental inequality
    \begin{equation}
\label{eqn:fundamental}
      \left\lvert\int_s^t Z_u\,\mathrm d\bX^a_u-\sum_{h\in\cF_{(N)}^<}Z_s^h\bX^{[h]_a}_{s,t}
      \right\rvert\le C\cnorm{\bZ}_\alpha\cnorm{\bX}_\alpha|t-s|^{(N+1)\alpha}
    \end{equation}
    Moreover, it defines an element \(\mathfrak I_\bX^a(\bZ)\in\cH_{(N)}\) with components
    \[
      \left\langle\bm 1,\mathfrak I_\bX^a(\bZ)\right\rangle=\int_0^tZ_u\,\mathrm d\bX^a_u,\quad\left\langle
      [h]_a,\mathfrak I_\bX^a(\bZ)\right\rangle = Z_t^h
    \]
    and zero otherwise.
    \begin{prop}
      \label{prp:int.bounded}
      The map \(\mathfrak I_\bX^a\colon\D{\alpha}{\bX}\to\D{\alpha}{\bX}\) is bounded, i.e., there is a constant
      \(C=C(\alpha)\) such that
      \[
        \cnorm{\mathfrak I_\bX^a(\bZ)}_\alpha\le C(1+T^\alpha)(1+\cnorm{\bX}_\alpha)\cnorm{\bZ}_\alpha.
      \]
    \end{prop}
    \begin{proof}
      For all \(h\in\cF_{(N-1)}\) set
      \[
        \mathfrak R^h_{s,t}\coloneq\delta\mathfrak I_\bX^a(\bZ)_{s,t}-\sum_{\bar h\in\cF_{(N-|h|-1)}^+}\mathfrak
        I_\bX^a(\bZ)^{\bar h\star h}_s\bX^{\bar h}_{s,t}.
      \]
      We first show that for all \(h\in\cF_{(N-2)}\) we have
      \[
        \mathfrak R^{[h]_a}_{s,t}=R^{h}_{s,t}+\sum_{\bar h\in\cF_{N-|h|-1}}Z^{\bar h\star h}_s\bX^{\bar h}_{s,t}
      \]
      and zero otherwise.
      Indeed, recall that for any \(h\in\cF_{(N-1)}\), we have that
      \[
        \mathfrak I_\bX^a(\bZ)^{[h]_a}_s=Z^h_s
      \]
      We note that by \Cref{lem:brace}, the identity
      \[
        \left\langle\bar h\star [h]_a,\mathfrak I_\bX^a(\bZ)_s\right\rangle=\left\langle [\bar h\star h]_a,\mathfrak
        I^a_\bX(\bZ)_s\right\rangle=Z^{\bar h\star h}_s
      \]
      holds.
      Hence
      \begin{align*}
        \mathfrak R^{[h]_a}_{s,t}&= \delta Z^h_{s,t}-\sum_{\bar h\in\cF_{(N-|h|-2)}^+}Z^{\bar h\star h}_s\bX^{\bar h}_{s,t}\\
        &= R^h_{s,t}+\sum_{\bar h\in\cF_{N-|h|-1}}Z^{\bar h\star h}_s\bX^{\bar h}_{s,t}.
      \end{align*}

      Now, given \(h\in\cF_{(N-2)}\), the remainder satisfies
      \begin{equation*}
        \|\mathfrak R^{[h]_a}\|_{(N-|h|-1)\alpha}\le \|R^h\|_{(N-|h|)\alpha}T^\alpha+\cnorm{\bX}_\alpha\sum_{\bar h\in\cF_{N-|h|}}\|Z^{\bar h\star h}\|_{\infty}.
      \end{equation*}
      Moreover, from the fundamental estimate \eqref{eqn:fundamental} we see that
      \[
        \|\mathfrak R^{\bm 1}\|_{N\alpha}\le C\cnorm{\bZ}_\alpha\cnorm{\bX}_\alpha
        T^\alpha+\cnorm{\bX}_\alpha\sum_{h\in\cF_{N-1}}\|Z^h\|_{\infty}.
      \]

      For any \(h\in\cF_{N-1}\) we have that
      \[
        \|Z^h\|_{\infty}\le|Z^h_0|+T^\alpha\|Z^h\|_{\alpha}.
      \]

      Finally, the norm can be bounded:
      \begin{align*}
        \cnorm{\mathfrak I_\bX^a(\bZ)}_\alpha&=\sum_{h\in\cF_{(N)}^<}\|\mathfrak R^h\|_{(N-|h|)\alpha}\\
        &= \|\mathfrak R^{\bm 1}\|_{N\alpha}+\sum_{h\in\cF_{(N-1)}^<}\|\mathfrak
        R^{[h]_a}\|_{(N-|h|-1)\alpha}\\
        &\le C\cnorm{\bZ}_\alpha\cnorm{\bX}_\alpha
        T^\alpha+\cnorm{\bZ}_\alpha T^\alpha+\cnorm{\bX}_\alpha\sum_{h\in\cF_{(N)}^<}(|Z^h_0|+\|Z^h\|_{\alpha})\\
        &\le C(1+\cnorm{\bX}_\alpha)(1+T^\alpha)\cnorm{\bZ}_\alpha.\qedhere
      \end{align*}
    \end{proof}
    \begin{prop}\label{cont:int_map}
      Let \(\bX,\tilde\bX\in\BRP^\alpha\) such that \(\cnorm{\bX}_\alpha\vee\cnorm{\tilde{\bX}}_\alpha\le M\) and let \(\bZ\in\D{\alpha}{\bX}\), \(\tilde\bZ\in\D{\alpha}{\tilde\bX}\) be such that \(\cnorm{\bZ}_\alpha\vee\cnorm{\tilde{\bZ}}_\alpha\le M\).
      Then there is a constant $C = C(\alpha,M)$ such that for all \(a\in A\), the bound
      \[
        \cnorm{\mathfrak I_\bX^a(\bZ);\mathfrak I^a_{\tilde\bX}(\tilde\bZ)}_\alpha\le
        C(\cnorm{\bZ;\tilde\bZ}_\alpha+\rho_\alpha(\bX,\tilde\bX))
      \]
      holds uniformly.
    \end{prop}
    \begin{proof}
      By considering the germ
      \[
        \Xi_{s,t}=\sum_{h\in\cF^<_{(N)}}Z^h_s\bX^{[h]_a}_{s,t}-\sum_{h\in\cF^<_{(N)}}\tilde{Z}^h_s\tilde{\bX}^{[h]_a}_{s,t}
      \]
      it is possible to show, using the Sewing Lemma \cite[Lemma 4.2]{FaH20}, that (using the same notations as in \Cref{prp:int.bounded}) 
      \[
        \|\mathfrak{R}^{\bm 1}-\tilde{\mathfrak{R}}^{\bm1}\|_{N\alpha}\le CT^\alpha(\cnorm{\bZ;\tilde{\bZ}}_\alpha+\rho_\alpha(\bX,\tilde{\bX}))+M\rho_\alpha(\bX,\tilde\bX)+M\sum_{h\in\cF_{N-1}}\|Z^h-\tilde{Z}^h\|_{\infty}.
      \]
      Proceeding in a similar way to the proof of \Cref{prp:int.bounded}, we see that
      \[
        \mathfrak R^{[h]_a}_{s,t}-\tilde{\mathfrak{R}}^{[h]_a}_{s,t}=R^h_{s,t}-\tilde{R}^h_{s,t}+\sum_{\bar h\in\cF_{N-|h|-1}}Z^{\bar h\star h}_s\bX^{\bar h}_{s,t}-\sum_{\bar h\in\cF_{N-|h|-1}}\tilde{Z}^{\bar h\star h}_s\tilde{\bX}^{\bar h}_{s,t}
      \]
      from where the bound
      \[
        \|\mathfrak{R}^{[h]_a}-\tilde{\mathfrak{R}}^{[h]_a}\|_{(N-|h|-1)\alpha}\le\|R^h-\tilde{R}^h\|_{(N-|h|)\alpha}T^\alpha+M\rho_\alpha(\bX,\tilde\bX)+M\sum_{\bar h\in\cF_{N-|h|-1}}\|Z^h-\tilde{Z}^h\|_{\infty}
      \]
      follows.
      Summing over \(h\in\cF_{(N)}^<\) yields the desired bound.
    \end{proof}
  \end{numba}

  \section{Main approximation result}\label{sec:main}
  In this section, we generalise a key result obtained in \cite{GRS22}. There it was shown in \cite[Theorem 3.10]{GRS22} that for a fixed $\alpha$-rough path $\mathbf{X}$, if $\beta < \alpha < 1/2$ are sufficiently close to $1/2$, the set
  $$\left\{(\psi,\psi')\, \middle|\, \psi_{s,t} = \int_s^t f_r\, \mathrm{d}X_{r} + \delta g_{s,t},\, \psi_s' = f_s\ \text{where } f, g \in C^\infty_1\right \},$$
  is dense in $\mathscr{D}^{\alpha,\beta}_{\mathbf{X}}$ (where the integral is understood in the Young sense). As a consequence, if $\mathbf{X}(\omega)$ is a random rough path of $\alpha$-regularity (e.g. the lift of a Brownian motion), the spaces $\{\mathscr{D}^{\alpha,\beta}_{\mathbf{X}(\omega)}\}_{\omega \in \Omega}$ turned out to be a measurable field of Banach spaces \cite[Proposition 3.15]{GRS22}. Our aim is to generalise and strengthen these results. We shall see that the field is indeed a \emph{continuous} field, we will remove the cumbersome conditions imposed on $\alpha,\beta$ in loc.cit.\, and obtain the result for arbitrary $\beta <\alpha$. To this end, we will construct a dense subset of smooth functions for every order:

\begin{defn}
  Fix \(\varepsilon\in(0,1-N\alpha)\) (where we recall that $N\alpha < 1$ due to our assumption that $\alpha \neq 1/n, n\in \N$) and \(\FnSp\) denote the closure of \(C^\infty\) under the
  \((1-\varepsilon)\)-Hölder norm.\footnote{These spaces are also known as the little Lipschitz spaces, see \cite[Chapter 4]{Weaver18} and cf.\ \cite[Theorem 5.33]{FV2010}} We recall that this space is a separable Banach space, and the following inclusions
  hold: let \(\mathrm{PL}\subset\Lip\) denote the space of piecewise linear functions on \([0,T]\), then
  \[
    \overline{\mathrm{PL}}^{\|\cdot\|_{1-\varepsilon}}=\FnSp\subset C^{1-\varepsilon}.
  \]
  Moreover, since \(1-\varepsilon>N\alpha>(N-1)\alpha>\dotsb>\alpha\) we get that
  \[
    \FnSp\subset C^{N\alpha}\subset\dotsb\subset C^\alpha
  \]
  and \(\|f\|_{k\alpha}\le\|f\|_{1-\varepsilon}T^{1-\varepsilon-k\alpha}\) for all \(k\in\{1,\dotsc,N\}\).

  Denote by \(\Omega\subset\FnSp\) a countable dense subset and define for $N \in \N$ the sets
\begin{align*}
  \mathcal{S}_N \coloneq \bigoplus_{h \in \mathcal{F}_{(N-1)}} \FnSp,\qquad \mathcal{S}_N^0
  \coloneq\bigoplus_{h\in\cF_{(N-1)}}\Omega.
\end{align*}
We endow \(\mathcal S_N\) with the norm
\[
  \fnorm{f}\coloneq\max_{h\in\cF_{(N-1)}}\|f^h\|_{1-\varepsilon}.
\]
\end{defn}

It will turn out that (rough) integration of elements in $\mathcal{S}_N$ leads to a dense subset of the Banach space of controlled paths over a given rough path $\bX$.
To establish this result, let us define a map \(\Gamma_{\bX}\colon\mathcal S_N\to\mathscr{D}^\alpha_\bX\) with the property that for all \(h\in\mathcal F_{(N-1)}\), the remainder
\begin{equation}\label{Rdefn}
  R^h_{s,t}=\delta\Gamma_\bX(f)^h_{s,t}-\sum_{\bar h\in\cF^+_{(N-|h|-1)}}\Gamma_\bX(f)^{\bar h\star h}_s\bX^{\bar h}_{s,t}
\end{equation}
satisfies \(|R^h_{s,t}|\le C(1+\cnorm{\bX}_\alpha)^{N-1-|h|}\fnorm{f}|t-s|^{1-\varepsilon}\).
The definition of $\Gamma_\bX (f)$ is recursive: 
\begin{equation}\label{GammaXtop}
\Gamma_\bX(f)_t^h\coloneq f^h_t \text{ for every \(h\in\mathcal F_{N-1}\).}
\end{equation}
Clearly,
\[
  \|R^h\|_{1-\varepsilon}=\|f\|_{1-\varepsilon}\le\fnorm{f}
\]
and in particular \(R^h\in C^{\alpha}_2\) with \(\|R^h\|_{\alpha}\le\fnorm{f}T^{1-\varepsilon-\alpha}\).

Given \(n<N-1\), let \(h\in\mathcal F_n\), and suppose that we have defined \(\Gamma_\bX(f)^{\bar h}\) for all forests
\(\bar h\in\mathcal F\) with \(n<|\bar h|\le N-1\), in a way such that \(|R^{\bar h}_{s,t}|\le C(1+\cnorm{\bX}_\alpha)^{N-1-|\bar h|}\fnorm{f}|t-s|^{1-\varepsilon}\).

Recall that every \(\rho\in\mathcal P^\top\) is of the form \(\rho=\top(p_{i_1},\dotsc,p_{i_k})\) for some integer \(k\ge 1\) and primitive elements \(p_{i_1},\dotsc,p_{i_k}\in\mathcal P\).
Given \(\rho\in\mathcal P^\top\) we define \(\rho^*=f_{i_1,\dotsc,i_k}\) to be its dual basis element.
\begin{lemma}
  \label{lem:gamma.int}
  Let \(h\in\cF_n\) and \(p\in\mathcal P\) be a primitive element with \(|p|<N-n\).
  Then the rough integral
  \[
    \int_0^t\Gamma_\bX(f)^{p^*\star h}_r\dd\bX^p_r\coloneq\lim_{|\pi|\to
    0}\sum_{[a,b]\in\pi}\sum_{\rho\in\cP^{\top,<}_{(N-n-|p|)}}\Gamma_\bX(f)^{\rho^*\star p^*\star h}_a\bX^{\rho\btop p}_{a,b}
  \]
  exists along any sequence of partitions, and it is independent of any choice.
  Moreover,
  \[
    \left\lvert\int_s^t\Gamma_\bX(f)^{p^*\star h}_r\dd\bX^p_{r}-\sum_{\rho\in\cP^{\top,<}_{(N-n-|p|)}}
      \Gamma_\bX(f)^{\rho^*\star p^*\star h}_s\bX^{\rho\btop p}_{s,t}\right\rvert\le
      C\fnorm{f}(1+\cnorm{\bX}_\alpha)^{N-2-n}\cnorm{\bX}_\alpha|t-s|^{1-\varepsilon+|p|\alpha}.
  \]
\end{lemma}
\begin{proof}
  Let
  \[
    \Xi_{s,t}\coloneq \sum_{\rho\in\mathcal P^{\top,<}_{(N-n-k)}}\Gamma_\bX(f)^{\rho^*\star p^*\star h}_s\bX^{\rho\btop p}_{s,t}.
  \]
  Observe that
  \begin{equation*}
    \delta\Xi_{s,u,t}= \sum_{\rho\in\mathcal P^{\top,<}_{(N-n-k)}}\left( \Gamma_\bX(f)^{\rho^*\star p^*\star h}_s\delta\bX^{\rho\btop p}_{s,u,t}-\delta\Gamma_\bX(f)^{\rho^*\star p^*\star h}_{s,u}\bX^{\rho\btop p}_{u,t} \right).
  \end{equation*}
  By recalling that
  \[
    \delta\bX^{\rho\btop p}_{s,u,t}=\bX^\rho_{s,u}\bX^p_{u,t}+\sum_{(\rho)}\bX^{\rho'}_{s,u}\bX^{\rho''\btop p}_{u,t},
  \]
  a standard computation gives
  \[
    \delta\Xi_{s,u,t}=-\sum_{\rho\in\mathcal P^{\top,<}_{(N-n-|p|)}}R^{\rho^*\star p^*\star h}_{s,u}\bX^{\rho\btop p}_{u,t}.
  \]
  Therefore
  \[
    |\delta\Xi_{s,u,t}|\le C\fnorm{f}(1+\cnorm{\bX}_\alpha)^{N-2-n}\cnorm{\bX}_\alpha\sum_{\rho\in\mathcal
      P^{\top,<}_{(N-n-|p|)}}|u-s|^{1-\varepsilon}|t-u|^{(|\rho|+|p|)\alpha}.
  \]
  The result then follows from the Sewing Lemma since \(1-\varepsilon+|p|\alpha>(N+|p|)\alpha>(N+1)\alpha>1\).
\end{proof}

We then set for every $h \in \mathcal{F}_n, 0 < n < N-1$
\begin{equation}\label{GammaXlower}
  \Gamma_\bX(f)^h_t\coloneq\sum_{\substack{p\in\mathcal P\\|p|<N-|h|}}\int_0^t\Gamma_\bX(f)^{p^*\star h}_r\dd\bX^p_r + f^h_t.
\end{equation}

\begin{lemma}
  \label{lem:gamma.lin.bound}
  The map \(\Gamma_\bX\colon\mathcal S_N\to\mathscr{D}^\alpha_\bX\) sending $f$ to $\Gamma_\bX(f)$ as in \cref{GammaXtop,GammaXlower} is well-defined, linear and bounded, with
  \begin{equation}\label{estimatefnomrvsbeta}
    \cnorm{\Gamma_\bX(f)}_\alpha\le C\frac{e^{N\cnorm{\bX}_\alpha}-1}{\cnorm{\bX}_\alpha}\fnorm{f}
  \end{equation}
  for all \(f\in\mathcal S_N\).
\end{lemma}
\begin{proof}
  The proof is by induction.
  Linearity clearly holds when \(|h|=N-1\). Indeed,
  \[
    \Gamma_\bX(f+\lambda g)^h_t=(f+\lambda g)^h_t=f_t^h+\lambda g^h_t=\Gamma_\bX(f)^h_t+\lambda\Gamma_\bX(g)^h_t.
  \]

  Now pick \(h\in\cF_{n}\), and suppose that linearity holds for all forests \(\bar h\) with \(|\bar h|>n\).  Then
  \begin{align*}
    \Gamma_\bX(f+\lambda g)^h_t&= \sum_{\substack{p\in\mathcal P\\|p|<N-n}}\int_0^t\Gamma_\bX(f+\lambda g)^{p^*\star h}_r\,\mathrm
    d\bX^p_r+(f+\lambda g)_t^h\\
    &= \sum_{\substack{p\in\mathcal P\\|p|<N-n}}\int_0^t\left\{ \Gamma_\bX(f)^{p^*\star h}_r +\lambda\Gamma_\bX(g)^{p^*\star
    h}_r\right\}\,\mathrm d\bX_r^p+f^h_t+\lambda g^h_t\\
    &= \Gamma_\bX(f)^h_t+\lambda\Gamma_\bX(g)^h_t.
  \end{align*}

  From the previous sections, we have that \(\|R^h\|_{\alpha}\le C\fnorm{f}T^{1-\varepsilon-\alpha}\) for all \(h\in\cF_{N-1}\).
  Take \(h\in\cF_n\), and note that
  \begin{align*}
    R^h_{s,t}&= \delta\Gamma_\bX(f)^h_{s,t}-\sum_{\bar h\in\cP^\top_{(N-n-1)}}\Gamma_\bX(f)^{\rho^*\star h}_s\bX^{\rho}_{s,t}\\
    &= \sum_{p\in\mathcal P_{(N-n-1)}}\int_s^t\Gamma_\bX(f)^{p^*\star h}_u\,\mathrm d\bX^p_u
      - \sum_{\rho\in\cP^\top_{(N-n-1)}}\Gamma_\bX(f)^{\rho^*\star h}_s\bX^{\rho}_{s,t}+\delta f^h_{s,t}\\
    &= \sum_{p\in\mathcal P_{(N-n-1)}}\left( \int_s^t\Gamma_\bX(f)^{p^*\star h}_u\,\mathrm d\bX^p_u 
      - \sum_{\rho\in\cP^\top_{(N-n-|p|)}}\Gamma_{\bX}(f)^{\rho^*\star p^*\star h}\bX_{s,t}^{\rho\btop p} \right) +
      \delta f^h_{s,t}.
  \end{align*}
  Therefore, by \Cref{lem:gamma.int} we see that
  \begin{equation*}
    |R^h_{s,t}|\le
    C\fnorm{f}(1+\cnorm{\bX}_\alpha)^{N-2-|h|}\cnorm{\bX}_\alpha\sum_{p\in\cP_{(N-n-1)}}|t-s|^{1-\varepsilon+|p|\alpha}+\|f^h\|_{1-\varepsilon}|t-s|^{1-\varepsilon},
  \end{equation*}
  so that
  \begin{equation}\label{shorttermestimate}
    \|R^h\|_{(N-n)\alpha}\le C_T(1+\cnorm{\bX}_\alpha)^{N-1-|h|}\fnorm{f}.
  \end{equation}
  since \(1-\varepsilon-(N-n+k)\alpha>(n-k)\alpha>(N-1)\alpha>0\) for all \(k\in\{0,\dotsc,N-n-1\}\).
  In particular, \(R^h\in C_2^{(N-|h|)\alpha}\) for all \(h\in\cF_{(N-1)}\), i.e.,
  \(\Gamma_\bX(f)\in\D{\alpha}{\bX}\), and the required bound holds after summation over \(h\in\cF_{(N)}^<\).
\end{proof}

\begin{theorem}
  \label{thm:lin.rho}
  Let \(\bX,\tilde\bX\in\BRP^\alpha\) such that \(\cnorm{\bX}_\alpha\vee\cnorm{\tilde{\bX}}_\alpha\le M\) and \(f\in\mathcal S_N\). Then,
  \[
    \cnorm{\Gamma_\bX(f);\Gamma_{\tilde\bX}(f)}_\alpha\le C\fnorm{f}\rho_\alpha(\bX,\tilde\bX).
  \]
\end{theorem}
\begin{proof}
  For \(h\in\cF_{(N-1)}\) and \(p\in\cP\), let us define
  \[
    Q^{h,p}_{s,t}\coloneqq\int_s^t\Gamma_{\bX}(f)^{p^*\star h}_u\,\mathrm d\bX^p_u-\sum_{\rho\in\cP^{\top,<}_{(N-|h|-|p|)}}\Gamma_\bX(f)^{\rho^*\star p^*\star h}_s\bX^{\rho\top p}_{s,t}.
  \]
  Using the Sewing Lemma and proceeding as in the proof of \Cref{lem:gamma.int} one can inductively show that
  \[
    |Q^{h,p}_{s,t}-\tilde Q^{h,p}_{s,t}|\le C\fnorm{f}\rho_\alpha(\bX,\tilde{\bX})|t-s|^{1-\varepsilon+|p|\alpha}.
  \]
  Then, arguing as in the proof of \Cref{lem:gamma.lin.bound}, we see that
  \[
    |R^h_{s,t}-\tilde{R}^h_{s,t}|\le C\sum_{p\in\cP_{(N-|h|-1)}}|Q^{h,p}_{s,t}-\tilde{Q}^{h,p}_{s,t}|,
  \]
  and the required bound is obtained after summation over \(h\in\cF_{(N)}^<\).
\end{proof}

In order the make the presentation more amenable to the reader, the construction of our approximation will be split in two parts.
First we show how to approximate a given controlled path \(\bZ\in\mathscr{D}^\alpha_{\bX}\) in the \(\cnorm{\cdot}_\beta\) norm over a small interval \([0,T]\), for any \(\beta<\alpha\).
Then, we extend this approximation to an arbitrary interval by patching together each of the individual approximations.

\subsection{Short-time approximations}\label{sec:short-time}

Given \(\bZ\in\D{\alpha}\bX\), we start by defining a controlled path of the form \(\Gamma_\bX(f)\colon[0,T]\to\mathcal F_{(N-1)}\)
on a given time interval $[0,T]$ such that \(\cnorm{\bZ;\Gamma_\bX(f)}_\beta\) is arbitrarily small for small $T > 0$ and any
\(\beta<\alpha\).

We do this by a backward procedure: The highest order functions are defined as affine functions. That is, for every \(h\in\mathcal F_{N-1}\), we set
\begin{align}\label{eqn:plinear}
  f^h_t = Z^h_0 + \frac tT(Z^h_T-Z^h_0),\quad t\in[0,T].
\end{align}
The idea is to define all lower order terms by integration against the rough path.

Assume that we have defined \(f^h\) for all \(h\in\cF\) with \(n<|h|\le N-1\).
Given \(h\in\cF_{n}\) we set
\begin{equation}
  f^h_t=Z^h_0+\frac{t}{T}\left[ \delta Z^h_{0,T}-\sum_{p\in\cP_{(N-|h|-1)}}\int_0^T\Gamma_\bX(f)^{p^*\star
  h}_u\,\mathrm d\bX^p_u \right].
  \label{defn:rec:path}
\end{equation}
where the integral is defined by \Cref{lem:gamma.int}.
Note that by construction we have 
\begin{align}\label{degreewiseestimate}
  \|f^h\|_{1-\varepsilon} \leq C\|R^h\|_{(N-|h|)\alpha}T^{\alpha (N-|h|)+\varepsilon-1}
\end{align}
for some $C>0$.

The controlled path $\Gamma_\bX(f)$ so constructed can serve as a first-order approximation to a given controlled path $\mathbf{Z}$. This is summarized in the next lemma.

\begin{lemma}\label{lemma:affine_approximation}
  Let $\mathbf{Z}$ be a controlled path. Define $\tilde\bZ\coloneq\Gamma_\bX(f)$ with $f$ as in \cref{eqn:plinear,defn:rec:path}   and $\tilde{R}^h$ as in \eqref{Rdefn}.
  Then $\tilde{\mathbf{Z}}$ has the following properties:
    \begin{itemize}
        \item[(i)] $\mathbf{Z}_0 = \tilde{\mathbf{Z}}_0$ and $\mathbf{Z}_T = \tilde{\mathbf{Z}}_T$.
        \item[(ii)] $\tilde{R}^h_{0,T} = R^h_{0,T}$ for every $h \in \mathcal{F}_{(N-1)}$.
        \item[(iii)] For every $\beta < \alpha$, there is a constant $C > 0$ such that
        \begin{align*}
            \|R^h - \tilde{R}^h \|_{(N - |h|) \beta} \leq CT^{(N - |h|)(\alpha - \beta)}.
        \end{align*}
        In particular,
 \begin{equation*}
  \cnorm{\bZ;\tilde\bZ}_\beta\le C\frac{T^{N(\alpha-\beta)}-T^{\alpha-\beta}}{T^{\alpha-\beta}-1}.
\end{equation*}
    \end{itemize}
\end{lemma}

\begin{proof}
    (i) follows by definition. Next,
    \begin{equation}
\label{eq:tildeR}
\begin{aligned}
  \tilde R^h_{0,T}&=  \delta\tilde Z^h_{0,T}-\sum_{\bar h\in\cF_{(N-|h|-1)}^+}\tilde Z^{\bar h\star h}_0\bX^{\bar h}_{0,T}\\
  &= \delta Z^h_{0,T} - \sum_{\bar h\in\cP_{(N-|h|-1)}^+} Z^{\bar h\star h}_0\bX^{\bar h}_{0,T} \\
  &= R^h_{0,T}
\end{aligned}
\end{equation}
and property (ii) is shown. Finally, in view of \eqref{degreewiseestimate} we obtain an estimate for $\fnorm{f}$ in terms of $\cnorm{\bZ}$. Thus property (iii) follows from the estimate \eqref{shorttermestimate} and the triangle inequality.

\end{proof}

\subsection{From local to global}\label{sec:local-global}
Having constructed affine controlled paths on a given interval, we will glue them together now. Fix a dissection \(\pi=\{0=t_0<t_1<\dotsb<t_N<t_{N+1}=T\}\) of \([0,T]\) with mesh \(\theta\coloneq\max_k|t_{k+1}-t_k|\). Set $I_k := [t_k,t_{k+1}]$. Assume that we are given controlled paths
\begin{align*}
    \tilde{\mathbf{Z}}^k \colon I_k \to \cH
\end{align*}
for which $\tilde{\mathbf{Z}}^k_{t_{k+1}} = \tilde{\mathbf{Z}}^{k+1}_{t_{k+1}}$ and with remainders satisfying $|\tilde{R}^{k;h}_{s,t}| \leq C |t-s| (t_{k+1} - t_k)^{(N - |h|)\alpha - 1}$ for every $s,t \in [t_k,t_{k+1}]$. We define $ \tilde{\mathbf{Z}} \colon [0,T] \to \cH$ by
\begin{align}\label{eqn:global_controlled_path}
  \tilde\bZ_t\coloneq\sum_{k=0}^N\tilde\bZ_t^k\mathbf 1_{I_k}(t)
\end{align}
and set
\begin{align*}
  \tilde{R}^h_{s,t} := \delta\tilde Z^h_{s,t} - \sum_{\bar h\in\cF_{(N-|h|-1)}^+}\tilde Z^{\bar h\star h}_s\bX^{\bar h}_{s,t}.
\end{align*}

\begin{lemma}\label{lemma:gubinelli83}
  For all \(h\in\mathcal F_{(N-1)}\) we have
  \[
    \delta \tilde{R}^h_{s,u,t} = \sum_{\bar h\in\cF_{(N-|h|-1)}^+} \tilde{R} _{s,u}^{\bar h\star h}\bX^{\bar h}_{u,t}.
  \]
  If \(\bZ\in\mathscr{D}^\alpha_{\bX}\), the same identity holds with $\tilde{R}$ replaced by the remainder $R$ of $\bZ$.
\end{lemma}

\begin{proof}
  By definition we have that
  \[
    \tilde R^h_{s,t}=\langle h, \delta\tilde\bZ_{s,t}\rangle - \langle\bX_{s,t}\otimes h,\Delta'\tilde\bZ_s\rangle
  \]
  Applying \(\delta\) to both sides of this equation and recalling that if \(f\in C_1\), \(g\in C_2\) and \(F_{s,t}\coloneq f_sg_{s,t}\in C_2\) then
  \[
    \delta F_{s,u,t}=f_s\delta g_{s,u,t}-\delta f_{s,u}g_{u,t}
  \]
  we see that
  \[
    \delta\tilde R^h_{s,u,t}=-\langle\bX_{s,u}\otimes\bX_{u,t}\otimes h,(\Delta'\otimes\id)\Delta'\tilde\bZ_s\rangle + \langle\bX_{u,t}\otimes h,\Delta'\delta\tilde\bZ_{s,u}\rangle
  \]
  Now, we have that for any \(f\in\cF_{(N-|h|-1)}\),
  \[
    \langle f\otimes h,\Delta'\delta\tilde\bZ_{s,u}\rangle=\sum_{\bar h\in\cF^0}\langle f\otimes h,\Delta'\bar h\rangle R_{s,u}^{\bar h} + \langle\bX_{s,u}\otimes f\otimes h,(\id\otimes\Delta')\Delta'\tilde\bZ_s\rangle
  \]
  and therefore
  \begin{align*}
    \delta\tilde R^h_{s,u,t}&= -\langle\bX_{s,u}\otimes\bX_{u,t}\otimes h,(\Delta'\otimes\id)\Delta'\tilde\bZ_s\rangle+ \langle\bX_{s,u}\otimes\bX_{u,t}\otimes h,(\id\otimes\Delta')\Delta'\tilde\bZ_s\rangle+\sum_{\bar h\in\cF^0}\langle \bX_{u,t}\otimes h,\Delta'\bar h\rangle\tilde R_{s,u}^{\bar h}\\
    &=  \sum_{\bar h\in\cF^0}\langle\bX_{u,t}\otimes h,\Delta'\bar h\rangle\tilde R_{s,u}^{\bar h}
  \end{align*}
  where the fist two terms cancel by coassociativity of \(\Delta'\).

  Finally, we have
  \begin{align*}
    \delta\tilde R^h_{s,u,t}&= \sum_{\bar h\in\cF_{(N-|h|-1)}^+}\sum_{f\in\cF^+_{(N-|h|-1)}}\langle f\otimes h,\Delta'\bar h\rangle\tilde R_{s,u}^{\bar h}\bX_{u,t}^f\\
    &= \sum_{f\in\cF^+_{(N-|h|-1)}}\tilde R_{s,u}^{f\star h}\bX_{s,u}^f\qedhere
  \end{align*}
\end{proof}

\begin{prop}\label{prop:constr_piecewise_linear_controlled}
    The path $\tilde \bZ$ defined in \eqref{eqn:global_controlled_path} is an $\alpha$-controlled path.
\end{prop}

\begin{proof}
 We have to show that $\| \tilde{R}^h \|_{(N - |h|)\alpha} < \infty$ for every $h \in \mathcal{F}_{(N-1)}$. Fix $s < t$ with $s \in [t_k,t_{k+1}]$ and $t \in [t_l,t_{l+1}]$. Using \Cref{lemma:gubinelli83} twice shows that
 \begin{align*}
   \tilde{R}^h_{s,t} = \tilde{R}^{k;h}_{s,t_{k+1}} + \tilde{R}^h_{t_{k+1}, t_l} + \tilde{R}^{l;h}_{t_l, t} + \sum_{\bar\in\cF_{(N-|h|-1)}^+} \tilde{R} _{s,t_{k+1}}^{k; \bar h\star h}\bX^{\bar h}_{t_{k+1},t} + \sum_{\bar h\in\cF_{(N-|h|-1)}^+} \tilde{R} _{t_{k+1},t_l}^{\bar h\star h}\bX^{\bar h}_{t_l,t}.
 \end{align*}
 By the triangle inequality, it suffices to show that $\tilde{R}^h_{s,t} \leq C|t-s|^{(N - |h|)\alpha}$ for every $s,t \in \pi$. This follows by induction over the length of $|t-s|$ using again \Cref{lemma:gubinelli83} for the induction step.
\end{proof}

\begin{prop}\label{prop:approximation_dyadic_controlled}
    Let $\bZ\in\mathscr{D}^\alpha_{\bX}$. Fix a dissection \(\pi_N=\{0=t_0<t_1<\dotsb<t_N<t_{N+1}=T\}\) of \([0,T]\) with mesh \(\theta\coloneq\max_k|t_{k+1}-t_k|\). On every interval $I_k = [t_k,t_{k+1}]$, we define $\tilde{\mathbf{Z}}^k \colon I_k \to \cH$ as in \Cref{lemma:affine_approximation} and $\tilde{\mathbf{Z}}$ as in \eqref{eqn:global_controlled_path}. Then
    \begin{align*}
        \cnorm{\bZ;\tilde\bZ}_\beta \to 0 \quad \text{as} \quad \theta \to 0
    \end{align*}
    for every $\beta < \alpha$.
\end{prop}

\begin{proof}
    By construction, \(\tilde\bZ\) satisfies \(\tilde\bZ_{t_k}=\bZ_{t_k}\) and \(\tilde R^h_{t_k,t_{k+1}}=R^h_{t_k,t_{k+1}}\) for all \(k=0,\dotsc,N\).
Moreover, for all \(h\in\mathcal F_{(N)}^+\) the remainders satisfy the bounds
\[
  \sup_{k=0,\dotsc,N}\sup_{s,t\in I_k}\frac{|\tilde R^h_{s,t}|}{|t-s|}\le C\theta^{(N-|h|)\alpha-1},\quad\sup_{k=0,\dotsc,N}\|R^h-\tilde R^h\|_{(N-|h|)\beta;I_k}\le C\theta^{(N-|h|)(\alpha-\beta)}.
\]
\Cref{lemma:gubinelli83} implies that for any pair of contiguous mesh points \(t_k<t_{k+1}\) and any \(t>t_{k+1}\) we have
\[
  \delta R^h_{t_k,t_{k+1},t}=\delta\tilde R^h_{t_k,t_{k+1},t}.
\]
Inductively, this implies that \(\tilde R^h_{t_j,t_k} = R^h_{t_j,t_k}\) for all \(0\le j<k\le N\), and so
\[
  \delta R^h_{t_j,t_k,t}=\delta\tilde R^h_{t_j,t_k,t}
\]
as well.
Consider now any two points \(s<t\in[0,T]\), and suppose that \(s\in I_j\), \(t\in I_k\) for some \(0\le j<k\le N\). Note that
\begin{align*}
  R^h_{s,t} - \tilde R^h_{s,t} &= R^h_{s,t_{j+1}} - \tilde R^h_{s,t_{j+1}} + R^h_{t_{j+1},t} - \tilde R^h_{t_{j+1},t} + \delta R^h_{s,t_{j+1},t} - \delta\tilde R^h_{s,t_{j+1},t}\\
  &= R^h_{s,t_{j+1}} - \tilde R^h_{s,t_{j+1}} + R^h_{t_{k},t} - \tilde R^h_{t_{k},t} + \delta R^h_{s,t_{j+1},t} - \delta\tilde R^h_{s,t_{j+1},t},
\end{align*}
hence, since both \(|t_{j+1}-s|\le|t-s|\) and \(|t-t_k|\le|t-s|\),
\begin{align*}
  |R^h_{s,t} - \tilde R^h_{s,t}|&\le \begin{multlined}[t]C\theta^{(N-|h|)(\alpha-\beta)}|t_{j+1}-s|^{(N-|h|)\beta}+C\theta^{(N-|h|)(\alpha-\beta)}|t-t_k|^{(N-|h|)\beta}\\
  +\sum_{\bar h\in\cF_{(N-|h|-1)}^+}|R^{\bar h\star h}_{s,t_{j+1}}-\tilde R^{\bar h\star h}_{s,t_{j+1}}||\bX^{\bar{h}}_{t_{j+1},t}|\end{multlined}\\
  &\le \begin{multlined}[t]2C\theta^{(N-|h|)(\alpha-\beta)}|t-s|^{(N-|h|)\beta}\\
  +C\sum_{\bar h\in\cF_{(N-|h|-1)}^+}\theta^{(N-|h|-|\bar h|)(\alpha-\beta)}|t_{k+1}-s|^{(N-|h|-|\bar h|)\beta}|t_{j+1}-t|^{|\bar h|\alpha}\end{multlined}\\
  &\le 3C\theta^{(N-|h|)(\alpha-\beta)}|t-s|^{(N-|h|)\beta}.
\end{align*}
Therefore, for all \(h\in\mathcal F_{(N-1)}\) and \(\beta\le\alpha\),
\[
  \|R^h-\tilde R^h\|_{(N-|h|)\beta}\le C\theta^{(N-|h|)(\alpha-\beta)}.
\]
Moreover,
\begin{align*}
  \cnorm{\bZ;\tilde\bZ}_\beta&= \sum_{h\in\cF_{(N)}^+}\|R^h-\tilde R^h\|_{(N-|h|)\beta}\\
  &\le C\sum_{h\in\cF_{(N)}^<}\theta^{(N-|h|)(\alpha-\beta)}\\
  &= C\sum_{k=1}^{N}\theta^{k(\alpha-\beta)}\to 0
\end{align*}
as \(\theta\to 0\).
\end{proof}

\begin{corollary}\label{cor:denseGamma}
  For every \(\bX\in\mathcal C^\alpha\), the linear subspace \(\Gamma_\bX(\mathcal S_N)\) is dense in \(\mathscr D^{\alpha,\beta}_\bX\) under \(\cnorm{\cdot}_\beta\) for \(\beta<\alpha\).
\end{corollary}

\begin{theorem}\label{thm:almost:cont:field}
 Let $N = \lfloor 1/ \alpha \rfloor$. Define the following sets of sections
    \begin{align*}
      \Gamma \coloneq\left\{ \bX\mapsto\Gamma_\bX(f):f\in\mathcal S_N \right\} ,\quad \Gamma_0\coloneq\left\{ \bX\mapsto\Gamma_\bX(f):f\in\mathcal S^0_N \right\}     \end{align*}
    Then the following holds:
    \begin{itemize}
        \item[(i)] $\Gamma$ is a linear subspace of $\prod_\bX\mathscr{D}_{\bX}^{\alpha,\beta} $
        \item[(ii)] For every $\mathbf{X}$, $\{\gamma(\mathbf{X})\, :\, \gamma \in \Gamma_0\}$ is a countable dense subset of $\mathscr{D}_{\bX}^{\alpha,\beta} $.
        \item[(iii)] For every $\gamma \in \Gamma$ with $\gamma(\mathbf{X}) = \Gamma_{\mathbf{X}}(f)$, the function $\mathbf{X} \mapsto \cnorm{\gamma(\mathbf{X})}_{\alpha} = \cnorm{\Gamma_{\mathbf{X}} (f)}_{\alpha}$ is continuous.
    \end{itemize}
\end{theorem}

\begin{proof}
  Part (i) follows from \Cref{lem:gamma.lin.bound} applied fiberwise, and part (ii) follows from \Cref{prop:approximation_dyadic_controlled} together with the observation that $\Gamma_0$ is countable as $\mathcal{S}_N^0$ is.

  Finally, we have to check continuity of the mapping $\mathbf{X} \mapsto \cnorm{\gamma(\mathbf{X})}_{\alpha}=\cnorm{\Gamma_{\mathbf{X}}(f)}_{\alpha}$. Since we are working in a metric space, it suffices to pick a sequence $\mathbf{X}_n \rightarrow \mathbf{X}$ as $\alpha$-H\"{o}lder rough paths. Then we apply \Cref{thm:lin.rho} to obtain
  $$|\cnorm{\Gamma_{\mathbf{X}_n}(f)}_\alpha - \cnorm{\Gamma_{\mathbf{X}}(f)}_\alpha|\le\cnorm{\Gamma_{\bX_n}(f);\Gamma_{\bX}(f)}_\alpha \leq C\fnorm{f}\rho_\alpha (\mathbf{X}_n,\mathbf{X}).$$
 Since $\mathbf{X}_n$ converges to $\mathbf{X}$ in the $\alpha$-H{\"o}lder topology, we see that the map $\mathbf{X} \mapsto \cnorm{\gamma (\mathbf{X})}_{\alpha}$ is continuous.
\end{proof}
 It is important to note that $\Gamma$ can be completed to yield a structure which is known as a \emph{continuous field of Banach spaces}, \cite{Dix77}. We will review the general theory of these structures in the next section.
 
 \begin{remark}[Approximation of controlled paths]\label{remark:smooth_approx}
 We will show now that our results immediately yield canonical approximations for controlled rough paths. 
Let $\mathbf{X}$ be an $\alpha$-rough path and assume that there are smooth rough paths $\mathbf{X}_{\varepsilon}$ (i.e. canonical lifts of smooth paths) such that $\mathbf{X}_{\varepsilon} \to \mathbf{X}$ as $\varepsilon \to 0$. Note that this is always the case when $\mathbf{X}$ is geometric. Let $\mathbf{Z} \in \mathscr{D}_{\bX}^{\alpha,\beta}$. We aim to construct smooth $\mathbf{Z}_{\varepsilon} \in \mathscr{D}_{\bX_{\varepsilon}}^{\alpha,\beta}$ such that $ \|\mathbf{Z};\mathbf{Z}_{\varepsilon}\|_{\beta} \to 0$ as $\varepsilon \to 0$. Let $\delta > 0$. From property (ii) in \Cref{thm:almost:cont:field}, we can choose $f \in \mathcal{S}_N$ such that
\begin{align*}
   \cnorm{ \mathbf{Z} ; \Gamma_{\mathbf{X}}(f) }_{\beta} = \cnorm{\mathbf{Z} - \Gamma_{\mathbf{X}}(f)}_{\beta} \leq \delta/2.
\end{align*}
From \Cref{thm:lin.rho}, we can choose $\varepsilon > 0$ sufficiently small to obtain
\begin{align*}
    \cnorm{ \Gamma_{\mathbf{X}}(f) ; \Gamma_{\mathbf{X}_{\varepsilon}}(f) }_{\beta} \leq \delta/2.
\end{align*}
Since $\mathbf{X}_{\varepsilon}$ is smooth, $\Gamma_{\mathbf{X}_{\varepsilon}}(f)$ is also smooth by construction. Therefore we can set $\mathbf{Z}_{\varepsilon} = \Gamma_{\mathbf{X}_{\varepsilon}}(f)$ and conclude with the triangle inequality that $\cnorm{ \mathbf{Z} ; \mathbf{Z}_{\varepsilon} }_{\beta} \leq \delta$.
 \end{remark}

\section{Banach bundles of controlled rough paths}
In this section we construct Banach bundles of controlled rough paths. We recall first some general definitions and results on Banach bundles and continuous fields of Banach spaces.

 \subsection{Continuous fields and Banach bundles}\label{sec:fobs}
 Let us first review two closely related structures which arose in conjunction with $C^\ast$-algebras and provide a convenient framework for our investigation of spaces of controlled rough paths. We shall now present these frameworks and discuss the pertinent examples from the theory of rough paths thereafter.

 \begin{defn}\label{defn_cont:field}
  Let $T$ be a topological space. A \emph{continuous field} $\Gamma$ of Banach spaces over $T$ is a family $(E_t)_{t\in T}$ of Banach spaces, together with a set $\hat{\Gamma} \subseteq \prod_{t \in T} E_t$ (where the elements of $\hat{\Gamma}$ are thought of as functions $\gamma \colon T \rightarrow \prod_t E_t$ with $\gamma(t) \in E_t, \forall t \in T$), such that:
  \begin{enumerate}
      \item[(i)] $\hat{\Gamma}$ is a linear subspace of $\prod_{t \in T} E_t$,
      \item[(ii)] for every $t \in T$ the set $\hat{\Gamma}(t) \coloneq \{\gamma (t) \mid \gamma \in \hat{\Gamma}\}$ is dense in $E_t$,
      \item[(iii)] For every $\gamma \in \hat{\Gamma}$, the function $T \rightarrow \R, t \mapsto \lVert \gamma (t)\rVert$ is continuous,
      \item[(iv)] Let $\tilde{\gamma} \in \prod_{t \in T} E_t$. If for every $t\in T$ and every $\varepsilon > 0$, there exists an $\gamma_t \in \hat{\Gamma}$ such that $\lVert \tilde{\gamma}(x)-\gamma_t(x)\rVert \leq \varepsilon$ on some neighborhood of $t$ in $T$, then $\tilde{\gamma} \in \hat{\Gamma}$.
  \end{enumerate}
  If moreover $\hat{\Gamma}$ contains a countable subset $\Lambda$ such that $\Lambda (t) = \{\gamma(t) \mid \gamma \in \Lambda\}$ is a dense subset in $E_t$ for every $t \in T$, then $\hat{\Gamma}$ is called a \emph{separable continuous field} of Banach spaces over $T$.
 \end{defn}
 In the previous section we have constructed an example of a family of sections of the spaces of controlled rough paths which is almost a continuous field of Banach spaces. Indeed our example satisfies properties (i)-(iii) but not (iv) of \Cref{defn_cont:field}. Let us agree to call such a collection a \emph{continuous pre-field} of Banach spaces. The reason the notion of pre-field does not exist as an independent object of study is that every pre-field can uniquely be completed to a continuous field. Namely, \cite[Proposition 10.2.3]{Dix77} yields

 \begin{prop}\label{prop:fieldcompletion}
   If $T$ is a topological space and $\Gamma$ a pre-field of Banach spaces over $T$. Then there exists a unique continuous field $\hat{\Gamma}$ of Banach spaces over $T$ such that $\Gamma \subseteq \hat{\Gamma}$. Furthermore, if $\Gamma$ admits a countable subset $\Lambda$ such that $\Lambda(t) = \{\gamma (t) \mid \gamma \in \Lambda\}$ is a dense subset of the Banach space over $t$, for all $t \in T$, then $\hat{\Gamma}$ is a separable continuous field of Banach spaces.
 \end{prop}

 \begin{proof}
  In view of the cited proposition, we only need to establish the separability. This is however trivial since we have assumed that there is a countable subset $\Lambda$ of $\Gamma \subseteq \hat{\Gamma}$ which yields the (countable) dense set $\Lambda(t)$ in each fibre for $t \in T$. Thus $\hat{\Gamma}$ is separable.
\end{proof}

The concept of a continuous fields is a convenient framework in which one can speak about continuous sections on a topological space with values in a collection of Banach spaces. The main point here is that the union of fibres $E \coloneq \sqcup_{t \in T} E_t$ does not need to carry a topology while the sections are still continuous in an appropriate sense. If one considered $E \rightarrow T$ as a bundle in the usual sense, it would be natural to define a topology on $E$ making the elements of a continuous field continuous. Indeed there exists a canonical way to introduce such a topology on $E$ from a given continuous field.

\begin{defn}\label{tubetopology}
 Let $\Gamma$ be a continuous (pre-)field of Banach spaces on a topological space $E$ and denote by $(E_t)_{t\in T}$ the family of Banach spaces over $T$. Define $E \coloneq \bigsqcup_{t \in T} E_t$ and let $p\colon E \rightarrow T$ be the canonical projection $p(e) = t$ if $e \in E_t$. Then we consider for $\gamma \in \Gamma$, $U \subseteq T$ open and $\varepsilon >0$ the \emph{tube}
 \begin{equation}\label{tubes}
 W (\gamma, U,\varepsilon) = \left\{b \in E\, \middle|\, p(b) \in U,\, \lVert b-\gamma (p(b))\rVert < \varepsilon \right\}
\end{equation}
From now on we endow $E$ with the topology generated by the base
$$\{W(\gamma,U,\varepsilon) \mid \gamma \in \Gamma, U \subseteq T \text{ open, and } \varepsilon >0 \}.$$
\end{defn}

 The point is that the tubes define a base of a Hausdorff topology on $E$. Before we discuss this, let us recall following concept.

\begin{defn}
Let $E,T$ be Hausdorff topological spaces and $p \colon E \rightarrow T$ be a continuous and open surjective map. We call the triple $\mathcal{E} = (B,p,T)$ a \emph{Banach bundle} over $T$ if each fibre $E_t \coloneq p^{-1}(t), t \in T$ is a Banach space and the following conditions hold
\begin{enumerate}
    \item[(i)] the map  $E \rightarrow \R, s \mapsto \lVert s\rVert$ is continuous.
    \item[(ii)] $+ \colon E_t \times E_t \rightarrow E,\quad (a_1,a_2) \mapsto a_1 + a_2$ is continuous for all $t\in T$,
    \item[(iii)] For each $\lambda \in \R,$ the map $E \rightarrow E, a \mapsto  \lambda a$ is continuous,
    \item[(iv)] if $t \in T$ and $\{a_i\}_{i \in I}$ is any net of elements\footnote{While the general definition of a Banach bundle is naturally phrased using nets, we will only encounter Banach bundles which are metrisable topological spaces. Thus all nets we have to care about are countable, i.e. sequences.} in $E$ such that $\lVert a_i\rVert\rightarrow 0$ and $p(a_i) \rightarrow t$ in $T$, then $a_i \rightarrow 0_t$ in $T$.
\end{enumerate}
Note that in the notation we usually suppress the index for the vector space operations and norms if the fibre is clear from the context. A continuous map $\sigma \colon T \rightarrow E$ such that $p\circ \sigma = \id_T$ is called \emph{continuous crossection}.
\end{defn}

\begin{remark}
A Banach bundle in the above sense is a special type of fibre bundle whose fibres are Banach spaces. It is more general than a vector bundle whose fibres are Banach spaces (and which is also often called a Banach bundle in the literature). Note that contrary to vector bundles, Banach bundles do in general not admit a local trivialisation by a bundle trivialisation. Consequently, we do not claim that a canonical trivialisation exists for the Banach bundle $B$.
\end{remark}

Banach bundles are (up to some technicalities) equivalent to continuous fields of Banach spaces as the next result shows:

\begin{prop}\label{prop:Banachbundleconst}
 Let $\Gamma$ be a continuous pre-field of Banach spaces over a metric space $(T,d)$. Let again $(E_t)_{t \in T}$ be the family of Banach fibres and endow $E\coloneq \sqcup_{t \in T} E_t$ with the topology from \Cref{tubetopology}. Then the canonical projection $p\colon E \rightarrow T$ turns $(E,p,T)$ into the unique Banach bundle such that the following holds:
 \begin{enumerate}
     \item Every element $\sigma \in \Gamma$ is a continuous cross section of the Banach bundle.
     \item The Banach bundle has enough cross sections, i.e.~for every $b \in E$ there exists a continuous cross section $\sigma_b$ with $\sigma_b (p(\sigma))=b$.
     \item If the completion $\hat{\Gamma}$ is a separable continuous field of Banach spaces, $E$ is a metric space. Moreover, $E$ is separable if and only if $T$ is separable.
 \end{enumerate}
\end{prop}
\begin{proof}
 By definition of a continuous pre-field of Banach spaces the set $\Gamma (t) = \{\gamma (t) \mid t \in T\}$ is a dense linear subspace of $E_t$ for each $t \in T$. Moreover, for every $\gamma \in \Gamma$, the function $t \mapsto \lVert \gamma(t)\rVert$ is continuous. Hence, the prerequisites of \cite[13.18 Theorem]{FaD88} are verified. Applying the theorem, we see that the tubes \eqref{tubes} form the basis of a Hausdorff topology on $E$. This topology turns $(E,p,T)$ into a Banach bundle such that every element $\sigma$ of $\Gamma$ becomes a continuous cross section of this bundle. Furthermore this topology is unique and satisfies (1) and has enough cross sections by \cite[Remark 13.19]{FaD88}, i.e.~(2) holds.

 We are left to prove (3) and note first that by uniqueness of the Banach bundle topology, we may replace $\Gamma$ with its completion $\hat{\Gamma}$. Hence without loss of generality we may assume now that $\Gamma$ is a separable continuous field of Banach spaces with a countable subset $\Lambda$ such that $\Lambda(t) = \{\gamma (t)\mid \gamma \in \Lambda\}$ is a dense subset of $E_t$ for each $t \in T$.
 To see that $E$ is metrisable we note first that, as a metric space, $T$ is paracompact. Hence we can apply \cite[Proposition A.1]{Laz18} and deduce that $E$ is a completely regular topological space. In particular, $E$ is a regular topological space. Further, \cite[Corollary 4.4.4]{Eng89} shows that $T$ admits a $\sigma$-finite topological base $\mathcal{B}$, i.e.\ a topological base which is a countable union of locally finite families of open sets $\mathcal{B} = \bigcup_{n\in \N} \mathcal{B}_n$.

 Let us construct a $\sigma$-finite topological base for $E$. For this let $\mathcal{B}_n = \{U_i\}_{i\in S_n}, n\in \N$ for the locally finite families of open sets $U_i$ comprising $\mathcal{B}$. Consider the family $\mathcal{T}$ of tubes defined via
 \begin{equation}\label{spec_tube}
     W(\gamma, U_i, r), \gamma \in \Lambda, i\in S_n, r \in \Q \cap (0,1).
 \end{equation}
 By construction $\mathcal{T}$ consists of open neighborhoods in $E$ and if $W(\gamma, U_i, r) \cap W(\gamma', U_i', s) \neq \emptyset$ we must have $U_i \cap U_i' \neq \emptyset$. As $\Lambda$ is countable, the set
 $$\mathcal{T} = \bigcup_{n \in \N} \bigcup_{r \in \Q \cap (0,1)}\bigcup_{\gamma \in \Lambda} \{W(\gamma,U_i,r)\}_{i \in S_n}$$
 is again $\sigma$-finite. We are left to prove that it is a topological base. Pick $W(\eta,U,r)$, for $\eta \in \Gamma, U\subseteq T$ open and some $r >0$. For any fixed $b \in W(\eta,U,r)$ we construct a tube of the form \eqref{spec_tube} contained in $W(\eta,U,r)$ which contains $b$. Pick a rational number $0<\varepsilon < \lVert b - \eta (p(b))\rVert /2$. Exploiting that the image $\Lambda(t)$ of $\Lambda$ is dense in $E_t$ for every $t \in T$, we can pick $\gamma_b \in \Gamma_0$ such that $\lVert b- \gamma_b(p(b))\rVert < \varepsilon$. Applying the triangle inequality we see that $\gamma_b (p(b)) \in W(\eta,U,r)$. Moreover, $t \mapsto \lVert \gamma_b(t)-\eta(t)\rVert$ is continuous, whence there is a small neighborhood $O$ of $p(b)$ in $U$ such that $W(\gamma_b,O,\varepsilon) \subseteq W(\eta,U,r)$. As $\mathcal{B}$ is a topological base of $T$, we can pick $i \in S_n$ for some $n\in \N$ such that $p(b)\in U_i \subseteq O$. This implies that $b \in W(\gamma_b, U_i,\varepsilon) \subseteq W(\eta,U,r)$. We conclude that $\mathcal{T}$ is a $\sigma$-finite topological base.

 Summing up, $E$ is a regular topological space with a $\sigma$-finite topological base, hence the Nagata-Smirnov metrisation theorem \cite[4.4.7]{Eng89} establishes that $E$ is metrisable.
 To establish the claim on separability, we recall that $(E,p,T)$ has enough cross sections and $T$ is a metric space. Metric spaces are second countable if and only if they are separable. Since $\Lambda (t)$ is dense in every $E_t$, \cite[Proposition 13.21]{FaD88} shows that the metric space $E$ is second countable (equivalently separable) if and only if $T$ is so. This establishes (3) and finishes the proof.
\end{proof}

\begin{remark}
Note that most of the statements in \Cref{prop:Banachbundleconst} are easy corollaries from well known results on Banach bundles. Hence they are hardly surprising. However, we were unable to locate the statement on the metrisability of the total space in the literature. (A metrisability statement for the total space under the more restrictive assumption that $T$ is separable can be found in \cite{Laz18}.)
\end{remark}

The uniqueness assertion in \Cref{prop:Banachbundleconst} shows that a continuous (pre-)field of Banach spaces uniquely determines a Banach bundle with enough sections. A deep result due to A.~Douady, L.~dal Soglio Herault and K.H.~Hofmann shows that this can be reversed. Indeed over paracompact topological spaces, there is a one-to-one correspondence between continuous fields of Banach spaces and Banach bundles which admit continuous cross sections through every point (see \cite[Appendix C]{FaD88}).
However, for later use (see \Cref{rem:equivalence}) it is important to note that the general construction only allows one to recover the continuous field, not the pre-field from which the field might have arisen by completion. 

We conclude this section with a brief discussion of the construction of the metric on the total space of the Banach bundle in \Cref{prop:Banachbundleconst}. This is essentially just a recap of the proof of the Nagata-Smirnov metrisation theorem. However, we need the cocnrete form of the metric constructed in the next section to study the Banach bundle of controlled rough paths.

\subsubsection{The Nagata-Smirnov metric on the Banach bundle}
Let $(E,p,T)$ be a Banach bundle over a metric space $(T,\rho)$ constructed from a separable continuous Banach bundle $\Gamma$ with countable dense subset $\Lambda$ via \Cref{prop:Banachbundleconst}. To understand convergence in the metric space $(E,d_E)$ we explicitly construct $d_E$ by revisit the Nagata-Smirnov metrisation theorem, \cite[4.4.7]{Eng89}. Its proof shows that the construction hinges on two choices:
\begin{enumerate}
    \item the $\sigma$-finite topological base $U_i, i \in S_n, n \in \N$. We note that since $(T,\rho)$ is a metric space, we can choose every $U_i = B_{r_i} (t_i)$  as a metric ball in $T$ with $t_i \in T$, and $0< r_i < 1$ (cf.\ \cite[Theorem 4.4.3]{Eng89}.
    \item For every tube in the basis of the topology one picks
    $$W(\gamma,B_{r_i}(t_i),\delta), i\in S_n, \gamma \in \Lambda , n \in N, \delta \in \mathbb{Q} \cap (0,1)$$
    a continuous function $f_{(\gamma,i,\delta)} \colon E \rightarrow [0,1]$ such that $W(\gamma,B_{r_i}(t_i),\delta)= f_{(\gamma,i,\delta)}^{-1}(]0,\infty[)$.
\end{enumerate}
Note that also here $\delta < 1$ is enough since the section ins $\Lambda$ have dense image in every fibre.
In the situation at hand, the following functions are valid choices
    \begin{align}\label{special_f}
    f_{(\gamma,i,\delta)} (b) \coloneq \begin{cases}
        (r_i - \rho (p(b),t_i))(\delta-\lVert b-\gamma(p(b))\rVert), & \text{if } b \in W(\gamma , B_{r_i}(t_i),\delta) \\
        0, & \text{else}
        \end{cases}.
        \end{align}
For later use we remark that both factors in the product in \eqref{special_f} are smaller then $1$ by choice of $r_i$ and $\delta$. By construction $|f_{(\gamma,i,\delta)}(x) - f_{(\gamma,i,\delta)} (y)|= 0$ if $(x,y) \not\in W(\gamma,U_i,\delta) \times E \cup E \times W(\gamma,U_i,\delta)$.
To ease notation, let us agree on the following definition:

\begin{defn}
For $\gamma \in \Gamma$, $i \in S_n, n\in \N$ such that $U_i = B_{r_i} (t_i)$ for some $r_i > 0$ and rational $\delta > 0$, we define
a pseudometric on $E$
\begin{align*}
    d_{(\gamma,i, \delta)}(x,y) = |f_{(\gamma,i,\delta)}(x) - f_{(\gamma,i,\delta)}(y)|.
\end{align*}
\end{defn}

Pick now an enumeration $m \mapsto (n_m,\gamma_m, \delta_m)$ of the countable set $\N \times \Lambda \times (\mathbb{Q} \cap ]0,\infty[)$ and recall that the open sets $U_i = B_{r_i} (t_i)$ for $i \in S_{n_m}$ form a locally finite family. Thus we obtain a well-defined continuous function
$$d_{m} \colon E \times E \rightarrow [0,\infty[ ,\quad (x,y) \mapsto \sum_{i\in S_{n_m}} d_{(\gamma_m,i, \delta_m)} (x,y).$$
Cutting off at $1$ and summing up, this yields the Nagata-Smirnov metric on $E$:
$$d_E\colon E \times E\rightarrow \R , d_E(x,y) \coloneq \sum_{m \in \N} 2^{-m} \min \{1,d_{m} (x,y)\}.$$
Note that convergence $x_n \rightarrow x$ in the metric $d_E$ implies $d_{m}(x_n,x) \rightarrow 0$ for all $m \in \N$.
As a direct consequence of the construction of the $d_m$ we thus obtain the following.

\begin{lemma}\label{lemma:convergence}
Let $(x_n)_{n\in \N}$ be a sequence in $(E,d_E)$. Then $x_n \rightarrow_{d_E} x$ if and only if
\begin{align*}
\lim_{k \rightarrow \infty} d_{(\gamma_m,i,\delta_m)} (x_k,x) = 0 \text{ for all } m\in \N, i \in S_{n_m}.
\end{align*}
\end{lemma}

Finally, we note that if $x,y \in W(\gamma, B_{r_i}(t_i),\delta)$ then an elementary estimate yields
\begin{align*}
    d_{(\gamma, i, \delta)}(x,y)  \leq& \underbrace{(r_i - \rho(p(x),t_i))}_{\in (0,1) }(\lVert x-\gamma(p(x))\rVert - \lVert y - \gamma(p(y))\rVert) \\
    &+ \underbrace{(\delta-\lVert y-\gamma(p(y))\rVert)}_{\in (0,1)}(\rho (p(y),t_i)-\rho(p(x),t_i))|
\end{align*}
Hence the pseudometric $ d_{(\gamma, i, \delta)}(x,y)$ is dominated by the sum of the two terms
\begin{align}
    |\lVert x-\gamma(p(x))\rVert - \lVert y - \gamma(p(y))\rVert| & \label{norm:triag}\\
    |\rho(p(y),t_i)-\rho(p(x),t_i)| &\leq \rho(p(y),p(x)) \label{metric:triag}
\end{align}
The expression \eqref{norm:triag} can only be simplified using the reverse triangle inequality (as we did in \eqref{metric:triag}) if $p(x)=p(y)$, i.e.\ if both points are contained in the same fibre.

For a converging sequence $x_n\rightarrow_{d_E} x$, where $x$ lies in the tube $W(\gamma, B_{r_i}(t_i),\delta)$, the above terms control convergence if enforced for all tubes in the base in which $x$ is contained. By \Cref{lemma:convergence} all terms \eqref{norm:triag} and \eqref{metric:triag} which are not cut off need to converge to $0$. Besides the estimate \eqref{metric:triag}, continuity of $p$ with respect to the tube topology yields $p(x_n) \rightarrow_\rho p(x)$ if $x_n\rightarrow_{d_E} x$.

 In the next section, we shall consider the Nagata-Smirnov metric for the Banach bundle of controlled rough paths. Our preparation here enables us to prove that it is equivalent to a metric constructed extrinsically on the union of the Banach fibres. This metric has been used extensively in the literature (see e.g. \cite{FaH20}) to establish stability of rough integration.

 \subsection{Properties of the Banach bundles of controlled rough paths}\label{sec:control-bundle}

We have seen in \Cref{thm:almost:cont:field} that the spaces of controlled rough paths form a continuous pre-field of Banach spaces over the space of branched rough paths. Moreover, we deduce from the result and \Cref{prop:fieldcompletion} the following:

\begin{corollary}\label{cor:field}
The continuous pre-field $\Gamma$ from \Cref{thm:almost:cont:field} can be uniquely completed to a separable continuous field of Banach spaces $\hat{\Gamma}$ over the space of branched $\alpha$-rough paths.
\end{corollary}

We will now leverage the theory recalled in the last section and consider two Banach bundles of controlled rough paths. For this let us fix some notation

\begin{defn}
 We denote by $\BRP^\alpha$ the \emph{space of branched rough paths} for some $0 <\alpha < 1$ and recall that it is a metric space with respect to the $\alpha$-rough path metric $\rho_\alpha$.
 Further, we let $\BRP^\alpha_g$ be the closed subspace of geometric rough paths endowed with the metric induced by the space of branched rough paths. For brevity we will also write $\rho_\alpha$ for the metric on this space.

 For any $0 <\beta < \alpha$ we define now the bundles
 \begin{align}
     p \colon E^{\alpha,\beta} \coloneq \bigsqcup_{\bX \in \BRP^\alpha} \mathscr{D}_\bX^{\alpha,\beta} \rightarrow \BRP^\alpha \label{branchedbundle} \\
     p \colon E^{\alpha,\beta}_g \coloneq \bigsqcup_{\bX \in \BRP^\alpha_g} \mathscr{D}_\bX^{\alpha,\beta} \rightarrow \BRP^\alpha_g \label{geometricbundle}
 \end{align}
 where $p$ is the canonical projection of a controlled rough path onto the controlling rough path.
\end{defn}

Combining now \Cref{cor:field} with \Cref{prop:Banachbundleconst} we see that the field of controlled rough paths induces a unique Banach bundle structure on the bundles \eqref{branchedbundle} and \eqref{geometricbundle}.
Moreover, we obtain the following result.

\begin{prop}\label{lem:metricstruct1}
The total spaces of the Banach bundles \eqref{branchedbundle} and \eqref{geometricbundle} are separable metric spaces with the Nagata-Smirnov metric. Moreover, $E^{\alpha,\beta}_g$ is separable  while $E^{\alpha,\beta}$ is not. 
\end{prop}

Having obtained the Banach bundles of controlled rough paths over the branched and geometric rough paths, let us introduce another metric which has been used in stability analysis of rough integrals (see e.g. \cite{FaH20})

\begin{defn}\label{defn:flatmetric}
 Let $E$ be either $E^{\alpha,\beta}$ or $E^{\alpha,\beta}_g$. We define a map $d^{\flat}_\alpha \colon E \times E \rightarrow \R$ with the help of \eqref{norm:controlled}. Namely for $x,y \in E$, we set 
\begin{align*}
    d^{\flat}_\alpha(x,y) \coloneq \rho_{\alpha}(p(x),p(y)) + \cnorm{x;y}_{\alpha}.
\end{align*}
In the following, whenever we derive results on $d^\flat_\alpha$ which hold for both Banach bundles we consider we shall denote their total spaces as $E$. 
\end{defn}

Note that $\cnorm{x;y}_\alpha$ makes sense for arbitrary elements $x,y \in E$ but it only constitutes a norm for elements belonging to the same fibre of the total space. The trick is that we can still compute a H\"{o}lder distance of the remainders as they take their values in the same Banach space. From the point of view of the Banach bundle this is however an extrinsic construction which does not reflect the geometric structure of the bundle (whence we chose to call this metric the \emph{flat} metric, as it exploits an extrinsic embedding into a flat space). 

\begin{lemma}
    The map $d^\flat_{\alpha}$ is a metric on the total space $E$, called the \emph{flat metric}.
\end{lemma}
\begin{proof}
We only have to prove that $d^\flat_{\alpha}(x,y) = 0$ implies $x = y$. To see this, we first note that $d^\flat_{\alpha}(x,y) = 0$ implies $p(x) = p(y)$. In this case, $\cnorm{x;y}_{\alpha} = \cnorm{x - y}_{\alpha}$ and since $\cnorm{\cdot}_{\alpha}$ is a norm on the space of paths controlled by $p(x)$, the claim follows.
\end{proof}

\begin{prop}\label{prop:metricsequivalent}
  The flat metric $d^\flat_{\alpha}$ and the metric $d_E$ are topologically equivalent. In other words, for a sequence $(x_n)_{n\in \N} \subseteq E$ we have $d^\flat_{\alpha}(x_n,x) \to 0$ as $n \to \infty$ if and only if $d_E (x_n,x) \to 0$ as $n \to \infty$.
\end{prop}
\begin{proof}
    Assume first that $\lim_{n \rightarrow \infty }d^\flat_{\alpha} (x_n,x) = 0$. By definition this implies $\rho_\alpha (p(x_n),p(x)) \rightarrow 0$ and the term \eqref{metric:triag} converges to $0$. For \eqref{norm:triag} we find with the negative triangle inequality
    \begin{align*}
    |\lVert x-\gamma(p(x))\rVert - \lVert x_n - \gamma(p(x_n))\rVert| \leq |d^\flat_{\alpha}(x, \gamma(p(x))) - d^\flat_{\alpha} (x_n, \gamma(p(x_n)))\rVert| \leq d^\flat_{\alpha}(x,x_n).
    \end{align*}
    Hence convergence in $d^\flat_{\alpha}$ implies convergence of \eqref{norm:triag} and \eqref{metric:triag} for all choices of $d_{(\gamma, i ,\delta)}$. Thus \Cref{lemma:convergence} shows that convergence with respect to $d^\flat_{\alpha}$ implies convergence in $d_B$. 
    
    For the converse, assume that $d_E (x_n,x) \to 0$. Since the bundle projection $p$ is a continuous map, this implies $\rho_{\alpha}(p(x_n),p(x)) \to 0$ as $n \to \infty$.
    
    Now let $\varepsilon > 0$. For every $\gamma \in \Gamma$ (note that for the bundle over the geometric rough paths, we replace $\Gamma$ by the restrictions of its elements to the subspace of geometric rough paths),
    \begin{align*}
        \cnorm{x;x_n}_{\alpha} \leq \cnorm{x - \gamma(p(x))}_{\alpha} + \cnorm{\gamma(p(x));\gamma(p(x_n))}_{\alpha} + \cnorm{\gamma(p(x_n) - x_n}_{\alpha}.
    \end{align*}
    We can choose $\gamma \in \Gamma_0$ such that $\cnorm{x - \gamma(p(x))}_{\alpha} \leq \varepsilon$. Continuity of the rough integral implies that we can find $M_1 > 0$ such that $\cnorm{\gamma(p(x));\gamma(p(x_n))}_{\alpha} \leq \varepsilon$ for every $n \geq M_1$. Choose an open neighborhood $U_i = B_{r_i} (\mathbf{X}_i)$ as in the definition of the Nagata-Smirnov metric (where $0 < r_i < 1$ and $\bX$ is a branched rough path for the bundle over the branched rough paths, otherwise we can pick a geometric rough path) such that $p(x) \in B_{r_i/2}(\mathbf{X}_i)$ and $p(x_n) \in B_{r_i}(\mathbf{X}_i)$ for every $n \in \N$ large enough. Since $x \in W(\gamma,B_{r_i}(\mathbf{X}_i),2\varepsilon)$, we must have $x_n \in W(\gamma,B_{r_i}(\mathbf{X}_i),2\varepsilon)$ for all $n \geq M_2$ and some $M_2$. We already know that $\rho_{\alpha}(p(x_n),p(x)) \to 0$ as $n \to \infty$ which implies that
    \begin{align*}
        (r - \rho_{\alpha}(p(x_n),\mathbf{X})) \to (r - \rho_{\alpha}(p(x),\mathbf{X}))
    \end{align*}
    as $n \to \infty$. Since $d_{(\gamma,i,\delta)}(x,x_n) \to 0$ for every rational $\delta >0$ as $n \to \infty$, this implies that
    \begin{align*}
        \| x_n - \gamma(p(x_n)) \|_{\alpha} \to \| x - \gamma(p(x)) \|_{\alpha}
    \end{align*}
    as $n \to \infty$. Since $\| x - \gamma(p(x)) \|_{\alpha} \leq \varepsilon$, we can conclude that $\| x_n - \gamma(p(x_n)) \|_{\alpha}  \leq 2 \varepsilon$ for $n \geq M_3$. Therefore, we have shown that for $n \geq M := \max\{M_1,M_2,M_3\}$,
    \begin{align*}
        \cnorm{x;x_n}_{\alpha} \leq 4 \varepsilon.
    \end{align*}
    Since $\varepsilon > 0$ was arbitrary, this concludes the proof.
\end{proof}

\begin{remark}\label{rem:equivalence}
In light of the equivalence of the Nagata-Smirnov metric and the flat metric, it is not hard to see that the Banach bundle structure can also be established using the flat metric. Hence from the flat metric, we could have constructed the Banach bundle and the associated continuous field. Now the reader may wonder whether the constructions in \cite{FaH20} (for level $2$-rough paths) would not lead to a direct proof of the articles results so far. While we can recover the continuous field by the general result, the same is not true for the pre-field of Banach spaces constructed through the approximation arguments. The point is that by working with the algebraic structure and the smaller set of sections from the pre-field our arguments yield estimates and control which can not (without further arguments) be deduced from the flat metric and the (non-canonical) homeomorphism with a trivial Banach bundle introduced in \cite{FaH20}.
\end{remark}
Note that the equivalence of the metrics is quite weak as they only induce the same topology. There are stronger notions of metric equivalence which ensure that properties established for one metric also carry over to the other. We will briefly discuss this now.
\begin{remark}
Topologically equivalent metrics do not need to have the same Cauchy-sequences (whence completeness of one metric is not automatically inherited by a topologically equivalent metric). For this, \emph{uniform equivalence of metrics} is sufficient. This means that the identity map $\id_X \colon (X,d_1) \rightarrow (X,d_2)$ and $\id_X \colon (X,d_2) \rightarrow (X,d_1)$ is uniformly continuous (and not just continuous as in topological equivalence). Unfortunately, without additional properties on the open cover uniform estimates of one metric against the other seem to be out of reach. The problem lies in the construction of the functions $d_m$ which are defined as sums of pseudometrics. Locally, every $d_m$ is a finite sum of such pseudometrics. However, since the number of summands can potentially be unbounded over the whole space (this depends on the chosen cover of the space of rough paths), this makes a uniform estimate impossible without further knowledge.  
\end{remark}

Consequently, we were only able to establish the completeness for the flat metric (completeness for the Nagata-Smirnov metric runs into a conceptually similar problem as a proof of the uniform equivalence). Note however, that the proof leverages the topological equivalence with the Nagata-Smirnov metric.

\begin{theorem}\label{thm:total_space_polish}
 Let $(E,d_\alpha^\flat)$ be the total space of the Banach bundle \eqref{branchedbundle} or \eqref{geometricbundle}, then $E$ is complete. In particular, the space $(E^{\alpha,\beta}_g,d_\alpha^\flat)$ is a Polish space while $(E^{\alpha,\beta},d_\alpha^\flat)$ is a complete metric space. 
\end{theorem}

\begin{proof}
In view of \Cref{lem:metricstruct1}, we know that the topological space is separable only for the bundle $(E^{\alpha,\beta}_{g},d_E)$. Now the Nagata-Smirnov metric and the flat metric are topologically equivalent by \Cref{prop:metricsequivalent}, whence the (non-)separability is inherited by the flat metric. Thus we only have to establish completeness of the space $(E,d^\flat_\alpha)$. For this we consider a Cauchy-sequence $(x_n)_{n \in \N}$ with respect to the flat metric. Recall from \Cref{defn:flatmetric} that being Cauchy in the flat metric implies that the sequence $(p(x_n))_{n\in \N}$ of basepoints must be a Cauchy-sequence with respect to the metric $\rho_\alpha$. Now the space of branched (respectively geometric) rough paths is complete with respect to $\rho_\alpha$, whence there exists a rough path $\bX$ such that $\lim_{n \rightarrow \infty} p(x_n) = \bX$. Furthermore, \Cref{defn:flatmetric} implies that for every $h \in \mathcal{F}^{<}_{(N)}$, the series $(\langle h, x_n \rangle)_n$ is a Cauchy sequence with respect to the uniform topology. By completeness, there exist continuous paths $x^h$ such that $\langle h, x_n \rangle \to x^h$ uniformly. Define $\mathbf{Z} \colon [0,T] \to \mathcal{H}^{<}_{(N)}$ by $\langle h, \mathbf{Z}_t \rangle = x^h_t$. We show that $\mathbf{Z}$ is controlled by $\mathbf{X}$. Define
\begin{align*}
    R^{h;n}_{s,t} &:= \langle h, x_n(t) \rangle - \langle p(x_n)_{s,t} \star h, x_n(s) \rangle \quad \text{and} \\ 
    R^{h}_{s,t} &:= \langle h, \mathbf{Z}_t \rangle - \langle \mathbf{X}_{s,t} \star h, \mathbf{Z}_s \rangle.
\end{align*}
Let $\varepsilon > 0$. Since $(x_n)_n$ is a Cauchy sequence w.r.t. $d^\flat_\alpha$, there is an $M$ such that
\begin{align*}
    \frac{|R^{h;n}_{s,t} - R^{h;m}_{s,t}|}{|t-s|^{(N - |h|)\alpha}} \leq \varepsilon 
\end{align*}
for every $s \neq t$ and every $n,m \geq M$. Letting $m \to \infty$, pointwise convergence implies that
\begin{align*}
    \|R^{h;n} - R^h\|_{(N - |h|)\alpha} \leq \varepsilon
\end{align*}
for $n \geq M$. This implies that $\mathbf{Z}$ is indeed controlled by $\mathbf{X}$ and that $d^\flat_\alpha(x_n,\mathbf{Z}) \to 0$ as $n \to \infty$ which proves completeness.
\end{proof}

It is now easy to recast stability results from the theory of rough paths in the language of Banach bundles (or equivalently continuous fields of Banach spaces).

\begin{prop}
Let again $E$ be the total space of the Banach bundle \eqref{branchedbundle} or \eqref{geometricbundle}. Then the integration map 
$$\mathfrak{I} \colon E \rightarrow E, \mathbb{Y} \mapsto \int \mathbb{Y} \mathrm{d}\bX, \text{ where } \bX = p(\mathbb{Y}).$$ 
is a morphism of Banach bundles over the identity, i.e.\ $\mathfrak{I}$ is continuous, fibrewise linear and projects via the bundle projection down to the identity.
\end{prop}

\begin{proof}
Let us note first that due to \Cref{goodintegrands} the integral of a controlled path is controlled again, whence the integration map makes sense and respect the fibres of the bundle. Since integration is linear in the integrand, we see that $\mathfrak{I}$ is fibre-wise linear. Now continuity with respect to the Banach bundle topology is equivalent to continuity with respect to $d_\alpha^\flat$ and this follows immediately from \Cref{cont:int_map}.
\end{proof}

Having established the integration map as a morphism of Banach bundles, we turn to the other prominent mapping in the setting, the It\^{o}-Lyons map which assigns to a rough differential equation its solution as a controlled rough path.

For this let us recall first that the definition depends on the choice of vector fields \(f_1,\dotsc,f_d\in C^\infty(\R^n,\R^n)\).
Following \cite{HK2015}, given \(\varphi\colon\R^n\to\R^n\) and \(v_1,\dotsc,v_m\in\R^n\) we denote
\[
  D^m\varphi(y):(v_1,\dotsc,v_m)\coloneq\sum_{\alpha_1,\dotsc,\alpha_m=1}^n\frac{\partial^m}{\partial
  y_{\alpha_1}\dotsm\partial y_{\alpha_m}}\varphi(y)v_1^{\alpha_1}\dotsm v_m^{\alpha_m}.
\]
Define the \emph{elementary differentials} \(f_\tau\in C^\infty(\R^n,\R^n)\) for \(\tau\in\cT\) recursively by \( f_{\bm1}(y)=y\) and
\[
  f_{[\tau_1\dotsm\tau_m]_i}\coloneq D^mf_i:(f_{\tau_1},\dotsc,f_{\tau_m}).
\]
\begin{lemma}\label{lem:Dn.tree}
  Let \(\tau,\rho_1,\dotsc,\rho_n\in\cT\), \(n\ge1\). Then
  \[
    D^nf_\tau:(f_{\rho_1},\dotsc,f_{\rho_n})=f_{\tau\curvearrowleft\rho_1\dotsb\rho_n}
  \]
\end{lemma}
\begin{proof}
  We first show the statement for the case \(n=1\) by induction on \(|\tau|\).
  The case when \(|\tau|=1\), i.e., when \(\tau=\Forest{[i]}\) follows from the defintion.
  Indeed,
  \[
    Df_i(y)f_\rho(y)=f_{[\rho]_i}(y)=f_{\Forest{[i]}\curvearrowleft\rho}(y).
  \]
  Now, if \(|\tau|=n+1\) there are \(\tau_1,\dotsc,\tau_m\in\cT\) with \(|\tau_1|+\dotsb+|\tau_m|=n\) and some label
  \(i\in\{1,\dotsc,d\}\), such that \(\tau=[\tau_1\dotsc\tau_m]_i=\Forest{[i]}\curvearrowleft\tau_1\dotsm\tau_m\).
  Hence, by the chain rule
  \begin{align*}
    Df_\tau:f_\rho&= D^{m+1}f_i:(f_{\tau_1},\dotsc,f_{\tau_m},f_{\rho})+\sum_{k=1}^mD^mf_i:(f_{\tau_1},\dotsc,Df_{\tau_k}f_\rho,\dotsc,f_{\tau_m})\\
    &=
    f_{[\tau_1\dotsm\tau_m\rho]_i}+\sum_{k=1}^mf_{[\tau_1\dotsm(\tau_k\curvearrowleft\rho)\dotsm\tau_m]_i}.
  \end{align*}

  Whence we conclude by noting that, at the level of trees, the identity
  \[
    [\tau_1\dotsm\tau_m\rho]_i+\sum_{k=1}^m[\tau_1\dotsm(\tau_k\curvearrowleft\rho)\dotsm\tau_m]_i=\Forest{[i]}\curvearrowleft(\tau_1\dotsm\tau_m\star\rho)=[\tau_1\dotsm\tau_m]_i\curvearrowleft\rho
  \]
  holds.

  Now, for the case \(n>1\), we note that by \cref{eq:arrow.star}
  \begin{align*}
    D^{n+1}f_\tau:(f_{\rho_1},\dotsc,f_{\rho_n},f_\rho)&=
    Df_{\tau\curvearrowleft\rho_1\dotsm\rho_n}:f_\rho-\sum_{k=1}^nD^nf_\tau:(f_{\rho_1},\dotsc,Df_{\rho_k}f_\rho,\dotsc,f_{\rho_n})\\
    &=
    f_{(\tau\curvearrowleft\rho_1\dotsm\rho_n)\curvearrowleft\rho}-\sum_{k=1}^nf_{\tau\curvearrowleft(\rho_1\dotsm(\rho_k\curvearrowleft\rho)\dotsm\rho_n)}\\
    &= f_{\tau\curvearrowleft\rho_1\dotsm\rho_n\rho}.\qedhere
  \end{align*}
\end{proof}

Given a branched rough path \(\bX\) and vector fields \(f_1,\dotsc,f_d\colon\R^n\to\R^n\), we say that \(Y\) solves the
Rough Differential Equation (RDE)
\[
  \mathrm dY_t=\sum_{i=1}^df_i(Y_t)\,\mathrm dX_t^i
\]
if it has the local expansion
\begin{equation}
\label{eq:sol.exp}
  \delta Y_{s,t}=\sum_{\tau\in\cT_{(N)}}\frac{1}{\sigma(\tau)}f_\tau(Y_s)\bX^\tau_{s,t}+r_{s,t},
\end{equation}
with \(r\in C_2^{(N+1)\alpha}\), where we recall that \(N=\lfloor\alpha^{-1}\rfloor\) so that \(N\alpha\le 1<(N+1)\alpha\), and
the path
\[
  \bY_t\coloneq \sum_{\tau\in\cT_{(N)}\cup\{\bm1\}}\frac{1}{\sigma(\tau)}f_\tau(Y_t)\tau
\]
belongs to \(\D{\alpha}{\bX}\).

It can be shown \cite{Gub10} that if \(f_i\in C_b^{N+1}\) then foreach \(\xi\in\R^n\), the RDE has a unique solution
\(\bY\coloneq\Phi(\xi,\bX)\) starting from \(Y_0=\xi\).
The map \(\Phi\colon\R^n\times \BRP^\alpha \to \bigsqcup_{\bX \in \text{BRP}^\alpha} \mathcal{D}^{\alpha}_\bX, (\xi,\bX) \mapsto \Phi (\xi,\bX)\) is known as the Itô-Lyons map.

\begin{prop}\label{thm:ItoLyonscts}
  The It\^{o}-Lyons map satisfies the bound
  \begin{align}\label{bound:ITOLYONS}
    \cnorm{\Phi(\xi,\bX);\Phi(\tilde\xi,\tilde\bX)}_\alpha\le C(|\xi-\tilde\xi|+\rho_\alpha(\bX,\tilde\bX)).
  \end{align}
 Hence, the It\^{o}-Lyons map is a parameter dependent continuous crossection of the Banach bundle $(E^{\beta,\alpha}_{\text{BRP}},p,\text{BRP}^\alpha)$ 
\end{prop}

Before proving the theorem, we recall the following result from \cite{Gub10,HK2015},
\begin{lemma}
  Let \(\bX\in\BRP^\alpha\) and \(\bZ\in\D\alpha\bX\). If \(\varphi\colon\R^n\to\R^n\) is a function of class \(C^N\)
  then the path \(\varphi(\bZ)\colon[0,T]\to\cH_{(N)}^{<}\) defined by \(\langle\bm1,
  \varphi(\bZ)_t\rangle\coloneq\varphi(Z_t)\) and
  \[
    \langle h,\varphi(\bZ)_t\rangle\coloneq\sum_{k=1}^{N-1}\sum_{\substack{f_1,\dotsc,f_k\in\cF\\f_1\dotsm
    f_k=h}}\frac{1}{k!}D^k\varphi(Z_t)(Z_t^{f_1},\dotsc,Z_t^{f_k})
  \]
  also belongs to \(\D\alpha\bX\).
  \label{lem:controlled.f}
\end{lemma}
\begin{proof}[Proof of \Cref{thm:ItoLyonscts}]

  Note that by definition $\Phi \colon \R^d \times  \BRP^\alpha \rightarrow E^{\alpha,\beta}$ satisfies $\Phi (\cdot, \bX) \in \D{\beta,\alpha}\bX$.
Hence we only need to establish continuity of $\Phi$.
For simplicity we only deal with the single noise case, i.e., \(|A|=1\) since the general case differs from this case
only in notation.
The proof follows ideas present in \cite{BFT2020}, in the setting of discrete rough paths.
First, note that the controlled path \(\bY=\Phi(\xi,\bX)\) solves the fixed-point equation
\[
  \bY_t=\bY_0+\mathfrak I_\bX(f(\bY))_t
\]
in \(\D\alpha\bX\).
Therefore, by \Cref{cont:int_map} we immediately obtain the bound
\[
  \cnorm{\Phi(\xi,\bX);\Phi(\tilde\xi,\tilde\bX)}_\alpha\le C\left( \cnorm{f(\bY);f(\tilde\bY)}_{\alpha}+\rho_\alpha(\bX,\tilde\bX) \right).
\]
From \Cref{lem:controlled.f} we see that for any forest \(h=\tau_1\dotsm\tau_k\in\cF_{(N)}^{<}\) we have
\begin{align*}
  \langle h, f(\bY)_t\rangle&= D^kf(Y_t)(\bY^{\tau_1},\dotsc,\bY^{\tau_k})\\
  &= D^kf(Y_t)(f_{\tau_1}(Y_t),\dotsc,f_{\tau_k}(Y_t))\\
  &= f_{[\tau_1\dotsm\tau_k]}(Y_t)
\end{align*}
Therefore, the remainder term can be expressed as
\begin{align*}
  R^h_{s,t}&= \delta f(\bY)_{s,t}^h-\sum_{\bar h\in\cF_{(N-|h|-1)}^+}f(\bY)_s^{\bar h\star h}\bX^{\bar h}_{s,t}\\
  &= f_{[h]}(Y_t)-f_{[h]}(Y_s)-\sum_{\bar{h}\in\cF_{(N-|h|-1)}^+}f_{[\bar{h}\star h]}(Y_s)\bX_{s,t}^{\bar{h}}.
\end{align*}
Now, we can perform a Taylor expansion on \(f_{[h]}\) and see that
\begin{equation*}
  f_{[h]}(Y_t)-f_{[h]}(Y_s)= \sum_{k=1}^{N-|h|-1}\frac{1}{k!}D^kf_{[h]}(Y_s)(\delta Y_{s,t})^{\otimes k}+T^h_{s,t}
\end{equation*}
where
\[
  T^h_{s,t}\coloneq\int_0^1\frac{1}{(N-|h|)!}D^{(N-|h|)}f_{[h]}(Y_s)(Y_s+\theta\delta Y_{s,t})^{\otimes(N-|h|)}(1-\theta)^{N-|h|-1}\,\mathrm d\theta.
\]
Thus, the remainder may be rewritten as
\[
  R^h_{s,t} = T^h_{s,t} + \sum_{k=1}^{N-|h|-1}\frac{1}{k!}D^kf_{[h]}(Y_s)(\delta Y_{s,t})^{\otimes k}-\sum_{\bar{h}\in\cF_{(N-|h|-1)}^+}f_{[\bar{h}\star h]}(Y_s)\bX_{s,t}^{\bar{h}}.
\]
Now, fix \(1\le k<N-|h|\). Replacing \cref{eq:sol.exp} into \(D^kf_{[h]}(Y_s)(\delta Y_{s,t})^{\otimes k}\) we obtain
\begin{align*}
  D^kf_{[h]}(Y_s)(\delta Y_{s,t})^{\otimes k}&= D^kf_{[h]}(Y_s)\left( \sum_{\tau\in\cT_{(N)}}Y^\tau_s\bX^\tau_{s,t} + r_{s,t} \right)^{\otimes k}\\
  &= k!\sum_{\bar{h}\in\cF_{(N-|h|-1)^+}}D^kf_{[h]\curvearrowleft\bar{h}}(Y_s)\bX^{\bar{h}}_{s,t}+B^h_{s,t},
\end{align*}
where \(B^h\) contains terms of the form \(D^kf_{[h]}:(f_{\rho_1},\dotsc,f_{\rho_\ell},R^{h_1}_{s,t},\dotsc,R^{h_m}_{s,t})\).

Proceeding similarly with the remainder corresponding to \(\tilde{\bY}\), we obtain that
\[
  |R^h_{s,t}-\tilde{R}^h_{s,t}|\le|T^h_{s,t}-\tilde{T}^h_{s,t}|+|B^h_{s,t}-\tilde{B}^h_{s,t}|.
\]
A straightforward bound gives
\[
  |T^h_{s,t}-\tilde{T}^h_{s,t}|\lesssim\left( \|Y-\tilde Y\|_{\alpha} P_{N-|h|}(\|Y\|_{\alpha},\|\tilde Y\|_\alpha)+\|Y-\tilde{Y}\|_\infty\|\tilde Y\|_\alpha^{(N-|h|)} \right)|t-s|^{(N-|h|)\alpha}
\]
with
\[
  P_m(x,y)\coloneq\sum_{j=1}^mx^{j}y^{m-j}.
\]
The difference \(B^h-\tilde{B}^h\) can be estimated in terms of the difference of remainders
\(R^{\bar{h}}-\tilde{R}^{\bar{h}}\) and the argument is closed recursively.

Due to the bound \eqref{bound:ITOLYONS} we see that $\Phi$ is continuous as a mapping into $(E^{\alpha,\beta},d_\alpha^\flat)$.
Now the flat metric generates the topology of the Banach bundle $E^{\alpha,\beta}$, \Cref{prop:metricsequivalent}, whence the continuity of $\Phi$ follows.
\end{proof}

\appendix

\section{Primitive elements and tree bases}\label{ap:prim}
This section presents tables of the primitive elements forests up to level four. We use these elements to generate a new basis by applying the natural growth operator (cf. \Cref{par:prim}). At the end, we apply our approximation procedure to an example. For calculating the primitive elements, we use the following operator defined in \cite[Theorem 9.6]{F2002},
\begin{align*}
&\pi_{1}:\mathcal{H}\rightarrow \mathrm{Prim},\\
& \pi_1 (h)=h-\sum_{(h)}h^{1}\btop \pi_{1}(h^2), 
\end{align*}
where $\Delta^{\prime}h=\sum_{(h)}h^{1}\otimes h^2$ and $\pi_1(\Forest{[]})=\Forest{[]}$. By \cite[Theorem 9.6]{F2002}, this operator is surjective. Accordingly, since this is also explicitly defined, we can quickly generate a basis for the space of primitive elements. The following vectors constitute a basis for this space up to level five.
\begin{center}
	\begin{tabular}{| m{2cm} || m{6cm}| m{6cm} | } 
		\hline
		Level (1)& $\pi_1(\Forest{[]})=\Forest{[]}$ & \\
		\hline
		Level (2)  & $\pi_1(\Forest{[]}\Forest{[]})=\Forest{[]}\Forest{[]}-2\Forest{[[]]}$  &   \\
		\hline
		Level (3)  & $\pi_1(\Forest{[]}\Forest{[]}\Forest{[]})=\Forest{[]}\Forest{[]}\Forest{[]}-3\Forest{[]}\Forest{[[]]}+3\Forest{[[[]]]}$  &   \\
		\hline
		Level (4)  & $\pi_1(\Forest{[]}\Forest{[]}\Forest{[]}\Forest{[]})=\Forest{[]}\Forest{[]}\Forest{[]}\Forest{[]}-4\Forest{[]}\Forest{[]}\Forest{[[]]}+4\Forest{[]}\Forest{[[[]]]}-4\Forest{[[[[]]]]}+4\Forest{[[]]}\Forest{[[]]}+2\Forest{[[][][]]}-4\Forest{[[][[]]]}+2\Forest{[[[][]]]}-2\Forest{[]}\Forest{[[][]]}$  & $\pi_{1}(\Forest{[]}\Forest{[]}\Forest{[[]]})= \Forest{[[][][]]}+\Forest{[[]]}\Forest{[[]]}-2\Forest{[[][[]]]}+\Forest{[[[][]]]}-\Forest{[]}\Forest{[[][]]}$\\
		\hline
	\end{tabular}
\end{center}

We can now apply to \Cref{Foissy_Basis} and find a new basis for $\mathcal{H}$. Here are the remaining vectors we need to add to construct a basis for the space of forests up to level four.

\begin{center}
\begin{tabular}{| m{2cm} || m{4cm}| m{4cm} | m{4cm} |  } 
	\hline
Level (2)& $\Forest{[]}\btop\Forest{[]}=\Forest{[[]]}$ &  & \\
\hline
Level (3)  & $\btop(\Forest{[]},\Forest{[]},\Forest{[]})=\Forest{[[[]]]}$  &  $\pi_1(\Forest{[]}\Forest{[]})\btop\Forest{[]}=\Forest{[[][]]}-2\Forest{[[[]]]}$ & $\Forest{[]}\btop\pi_{1}(\Forest{[]}\Forest{[]})=\Forest{[]}\Forest{[[]]}-\Forest{[[[]]]}-\Forest{[[][]]}$ \\
\hline
\end{tabular}
\end{center}
For the level (4),

\begin{center}
	\begin{tabular}{| m{4.8cm} | m{4.8cm}| m{4.8cm} |  } 
		\hline
		  $\btop\big{(}\Forest{[]},\Forest{[]},\pi_{1}(\Forest{[]}\Forest{[]})\big{)}=\Forest{[]}\Forest{[[[]]]}-\Forest{[[[]]]}-\Forest{[[][[]]]}$  &  $\btop\big{(}\Forest{[]},\pi_{1}(\Forest{[]}\Forest{[]}),\Forest{[]}\big{)}=,\Forest{[[][[]]]}-\Forest{[[[]]]}-\Forest{[[[][]]]}$ & $\btop\big{(}\pi_{1}(\Forest{[]}\Forest{[]}),\Forest{[]},\Forest{[]}\big{)}=\Forest{[[[][]]]}-2\Forest{[[[[]]]]}$    \\
		\hline
	 $\btop\big{(}\pi_{1}(\Forest{[]}\Forest{[]}),\pi_1(\Forest{[]}\Forest{[]})\big{)}=\Forest{[]}\Forest{[[][]]}-\Forest{[[[][]]]}-\Forest{[[][][]]}-2\Forest{[]}\Forest{[[[]]]}+2\Forest{[[[[]]]]}+2\Forest{[[][[]]]}$ & $\pi_{1}(\Forest{[]}\Forest{[]}\Forest{[]})\btop\Forest{[]}=\Forest{[[][][]]}-3\Forest{[[][[]]]}+3\Forest{[[[[]]]]}$&$\Forest{[]}\btop\pi_{1}(\Forest{[]}\Forest{[]}\Forest{[]})=\Forest{[]}\Forest{[]}\Forest{[]}\Forest{[[]]}-\Forest{[[]]}\Forest{[[]]}-\Forest{[]}\Forest{[[[]]]}-\Forest{[]}\Forest{[[][]]}+\Forest{[[[[]]]]}+\Forest{[[[][]]]}+\Forest{[[][[]]]}$
		\\
		\hline
			&$\btop(\Forest{[]},\Forest{[]},\Forest{[]},\Forest{[]})=\Forest{[[[[]]]]}$& 
		\\
		\hline
	\end{tabular}
\end{center}

We now apply our approximation scheme to the \Cref{ex:ctrl}.
  \begin{example}
Recall from the example that we consider a path \(\bZ\in\mathscr D_{\bX}^{\alpha}\) for \(\alpha\in(\tfrac15,\tfrac14)\). We can rewrite the controlledness condition of $\bZ$ with respect to $\bX$ in terms of the primitive basis as shown below:
\begin{align*}
		\delta Z^{\bm 1}_{s,t}&= \begin{multlined}[t]Z_s^{\tdot}\bX_{s,t}^{\tdot}+Z^{\tdot\tdot}_s\bX_{s,t}^{\tdot\tdot-2\tlI}+\left( 2Z_s^{\tdot\tdot}+Z_s^{\tlI}
			\right)\bX_{s,t}^{\tlI}+Z_s^{\tdot\tdot\tdot}\bX_{s,t}^{\tdot\tdot\tdot-3\tdot\tlI+3\tlII}\\
			+\left( 3Z_s^{\tdot\tdot\tdot}+Z_s^{\tdot\tlI}+Z_s^{\tv} \right)\bX_{s,t}^{\tv-2\tlII}+\left( 3Z_s^{\tdot\tdot\tdot}+Z_s^{\tdot\tlI}
			\right)\bX_{s,t}^{\tdot\tlI-\tlII-\tv}\\
			+\left( 6Z_s^{\tdot\tdot\tdot}+3Z_s^{\tdot\tlI}+2Z_s^{\tv}+Z_s^{\tlII} \right)\bX_{s,t}^{\tlII}+R^{\bm 1}_{s,t}
		\end{multlined}\\
		\delta Z_{s,t}^{\tdot}&= \begin{multlined}[t]
			\left( 2Z_s^{\tdot\tdot}+Z_s^{\tlI} \right)\bX^{\tdot}_{s,t}+\left( 3Z_s^{\tdot\tdot\tdot}+Z_s^{\tdot\tlI}+Z_s^{\tv}
			\right)\bX_{s,t}^{\tdot\tdot-2\tlI}\\
			+\left( 6Z_s^{\tdot\tdot\tdot}+3Z_s^{\tdot\tlI}+2Z_s^{\tv}+Z_s^{\tlII} \right)\bX_{s,t}^{\tlI}+R_{s,t}^{\tdot}
		\end{multlined}\\
		\delta Z_{s,t}^{\tdot\tdot}&= \left( 3Z_s^{\tdot\tdot\tdot}+Z_s^{\tdot\tlI}
		\right)\bX_{s,t}^{\tdot}+R_{s,t}^{\tdot\tdot},\\
		\delta Z_{s,t}^{\tlI}&= \left( Z_s^{\tdot\tlI}+Z_s^{\tlII}+Z_s^{\tv} \right)\bX_{s,t}^{\tdot}+R_{s,t}^{\tlI}.
	\end{align*}
Let us to fix a dissection \(\pi=\{0=t_0<t_1<\dotsb<t_N<t_{N+1}=T\}\) of \([0,T]\) and set $I_k := [t_k,t_{k+1}]$. For
$ s\in I_{k} $, we define
\begin{align*}
	&\tilde{Z}^{\Forest{[[[]]]}}_{s}={Z}^{\Forest{[[[]]]}}_{t_{k}}+\frac{s-t_k}{t_{k+1}-t_k}{Z}^{\Forest{[[[]]]}}_{t_{k},t_{k+1}}, \ \ \ \ \ \ \ \ \ \ \ \ \ \tilde{Z}^{\Forest{[]}\,\Forest{[[]]}}_{s}={Z}^{\Forest{[]}\,\Forest{[[]]}}_{t_{k}}+\frac{s-t_k}{t_{k+1}-t_k}{Z}^{\Forest{[]}\,\Forest{[[]]}}_{t_{k},t_{k+1}}\\
	&\tilde{Z}^{\Forest{[[],[]]}}_{s}={Z}^{\Forest{[[],[]]}}_{t_{k}}+\frac{s-t_k}{t_{k+1}-t_k}{Z}^{\Forest{[[],[]]}}_{t_{k},t_{k+1}}, \ \ \ \ \ \ \ \ \  \tilde{Z}^{\Forest{[]}\,\Forest{[]}\, \Forest{[]}}_{s}={Z}^{\Forest{[]}\,\Forest{[]}\, \Forest{[]}}_{t_{k}}+\frac{s-t_k}{t_{k+1}-t_k}{Z}^{\Forest{[]}\,\Forest{[]}\, \Forest{[]}}_{t_{k},t_{k+1}} .
\end{align*}
We then define
\begin{align*}
	\tilde{Z}^{\Forest{[]}\,\Forest{[]}}_{s} = {Z}^{\Forest{[]}\,\Forest{[]}}_{t_k} + &\int^{s}_{t_{k}}(3\tilde{Z}^{\Forest{[]}\,\Forest{[]}\,\Forest{[]}}_{\tau} + \tilde{Z}_{\tau}^{\Forest{[]}\,\Forest{[[]]}})\mathrm{d}\bX^{\Forest{[]}}_{\tau} + \frac{s-t_k}{t_{k+1}-t_k}\bigg{[}{Z}^{\Forest{[]}\,\Forest{[]}}_{t_{k},t_{k+1}} - \int^{t_{k+1}}_{t_{k}}(3\tilde{Z}^{\Forest{[]}\,\Forest{[]}\,\Forest{[]}}_{\tau} + \tilde{Z}_{\tau}^{\Forest{[]}\,\Forest{[[]]}})\mathrm{d}\bX^{\Forest{[]}}_{\tau}\bigg{]},
\end{align*}
and
\begin{align*}
	\tilde{Z}^{\Forest{[[]]}}_s={Z}^{\Forest{[[]]}}_{t_k}+&\int_{t_k}^{s}(\tilde{Z}^{\Forest{[]}\,\Forest{[[]]}}_{\tau}+2\tilde{Z}_{\tau}^{\Forest{[[][]]}}+\tilde{Z}_{\tau}^{\Forest{[[[]]]}})\mathrm{d}\bX^{\Forest{[]}}_{\tau}+\frac{s-t_k}{t_{k+1}-t_k}\bigg{[}{Z}^{\Forest{[[]]}}_{t_{k},t_{k+1}}-\int_{t_k}^{t_{k+1}}(\tilde{Z}^{\Forest{[]}\,\Forest{[[]]}}_{\tau}+2\tilde{Z}_{\tau}^{\Forest{[[][]]}}+\tilde{Z}_{\tau}^{\Forest{[[[]]]}})\mathrm{d}\bX^{\Forest{[]}}_{\tau}\bigg{]}.\\
\end{align*}
Finally, set
\begin{align*}
	\tilde{Z}^{\Forest{[]}}_{s}=Z_{t_k}^{\Forest{[]}}& + \int_{t_k}^{s}(2\tilde{Z}_{\tau}^{\Forest{[]}\,\Forest{[]}}+\tilde{Z}_{\tau}^{\Forest{[[]]}})\mathrm{d}\mathbf{X}_{\tau}^{\Forest{[]}} + \int_{t_k}^{s}(3\tilde{Z}_{\tau}^{\Forest{[]}\,\Forest{[]}\,\Forest{[]}}+\tilde{Z}_{\tau}^{\Forest{[[][]]}}+\tilde{Z}_{\tau}^{\Forest{[]}\,\Forest{[[]]}})\mathrm{d}\mathbf{X}_{\tau}^{\Forest{[]}\Forest{[]}-2\Forest{[[]]}}+\\&\frac{s-t_k}{t_{k+1}-t_k}\bigg{[}Z^{\Forest{[]}}_{t_{k},t_{k+1}}-\int_{t_k}^{t_{k+1}}(2\tilde{Z}_{\tau}^{\Forest{[]}\,\Forest{[]}}+\tilde{Z}_{\tau}^{\Forest{[[]]}})\mathrm{d}\mathbf{X}_{\tau}^{\Forest{[]}}-\int_{t_k}^{t_{k+1}}(3\tilde{Z}_{\tau}^{\Forest{[]}\,\Forest{[]}\,\Forest{[]}}+\tilde{Z}_{\tau}^{\Forest{[[][]]}}+\tilde{Z}_{\tau}^{\Forest{[]}\,\Forest{[[]]}})\mathrm{d}\mathbf{X}_{\tau}^{\Forest{[]}\Forest{[]}-2\Forest{[[]]}}\bigg{]},
\end{align*} 
and
\begin{align*}
	\tilde{Z}_{s}^{\bm 1} = Z_{t_{k}}^{\bm 1} & + \int_{t_k}^{s}\tilde{Z}^{\Forest{[]}}_{\tau}\mathrm{d}\mathbf{X}_{\tau}^{\Forest{[]}} + \int_{t_k}^{s}\tilde{Z}^{\Forest{[]}\,\Forest{[]}}_{\tau}\mathrm{d}\mathbf{X}^{\Forest{[]}\Forest{[]}-2\Forest{[[]]}}_{\tau} + \int_{t_k}^{s} \tilde{Z}^{\Forest{[]}\,\Forest{[]}\,\Forest{[]}}_{\tau}\mathrm{d}\mathbf{X}^{\Forest{[]}\Forest{[]}\Forest{[]}-3\Forest{[]}\Forest{[[]]}+3\Forest{[[[]]]}}_{\tau} + \\&\frac{s-t_k}{t_{k+1}-t_k}\bigg{[}Z_{t_{k},t_{k+1}}^{\bm 1}-\int_{t_k}^{t_{k+1}} \tilde{Z}^{\Forest{[]}}_{\tau}\mathrm{d}\mathbf{X}_{\tau}^{\Forest{[]}}-\int_{t_k}^{t_{k+1}}\tilde{Z}^{\Forest{[]}\,\Forest{[]}}_{\tau}\mathrm{d}\mathbf{X}^{\Forest{[]}\Forest{[]}-2\Forest{[[]]}}_{\tau}- \int_{t_k}^{t_{k+1}} \tilde{Z}^{\Forest{[]}\,\Forest{[]}\,\Forest{[]}}_{\tau}\mathrm{d}\mathbf{X}^{\Forest{[]}\Forest{[]}\Forest{[]}-3\Forest{[]}\Forest{[[]]}+3\Forest{[[[]]]}}_{\tau}\bigg{]}.
\end{align*}

\end{example}
\bibliographystyle{arxivalpha}
\bibliography{geometry_controlled_paths}

\newcommand{\etalchar}[1]{$^{#1}$}
\providecommand{\bysame}{\leavevmode\hbox to3em{\hrulefill}\thinspace}
\providecommand{\MR}{\relax\ifhmode\unskip\space\fi MR }
\providecommand{\MRhref}[2]{%
  \href{http://www.ams.org/mathscinet-getitem?mr=#1}{#2}
}
\providecommand{\href}[2]{#2}
\begin{thebibliography}{{Man}08}

\bibitem[Arn98]{A1998}
L.~Arnold, \href{http://dx.doi.org/10.1007/978-3-662-12878-7}{\emph{Random
  dynamical systems}}, Springer Monographs in Mathematics, Springer-Verlag,
  Berlin, 1998.

\bibitem[BBR{\etalchar{+}}20]{BaBaRaRaS20}
C.~{Bayer}, D.~{Belomestny}, M.~{Redmann}, S.~{Riedel}, and J.~{Schoenmakers},
  \emph{{Solving linear parabolic rough partial differential equations}}, {J.
  Math. Anal. Appl.} \textbf{490} (2020), no.~1, 44 (English), Id/No 124236.

\bibitem[BDS16]{BaDaS16}
G.~Bogfjellmo, R.~Dahmen, and A.~Schmeding, \emph{Character groups of {H}opf
  algebras as infinite-dimensional {L}ie groups},
  \href{http://dx.doi.org/10.5802/aif.3059}{Ann. Inst. Fourier (Grenoble)
  \textbf{66} (2016)}, no.~5, 2101--2155.

\bibitem[BFG16]{BFG2016}
C.~Bayer, P.~Friz, and J.~Gatheral, \emph{Pricing under rough volatility},
  \href{http://dx.doi.org/10.1080/14697688.2015.1099717}{Quant. Finance
  \textbf{16} (2016)}, no.~6, 887--904.

\bibitem[BFT20]{BFT2020}
C.~Bayer, P.~Friz, and N.~Tapia, \emph{Stability of deep neural networks via
  discrete rough paths}, 2020, \href{http://arxiv.org/abs/2732}{{\ttfamily WIAS
  Preprint:2732}}.

\bibitem[CK98]{CK1998}
A.~Connes and D.~Kreimer, \emph{Hopf algebras, renormalization and
  noncommutative geometry},
  \href{http://dx.doi.org/10.1007/s002200050499}{Comm. Math. Phys. \textbf{199}
  (1998)}, no.~1, 203--242.

\bibitem[CL01]{CL2001}
F.~Chapoton and M.~Livernet, \emph{Pre-{L}ie algebras and the rooted trees
  operad}, \href{http://dx.doi.org/10.1155/S1073792801000198}{Internat. Math.
  Res. Notices (2001)}, no.~8, 395--408.

\bibitem[DF88]{FaD88}
R.~S. {Doran} and J.~M.~G. {Fell}, \emph{{Representations of *-algebras,
  locally compact groups, and Banach *- algebraic bundles. Vol. 1: Basic
  representation theory of groups and algebras}}, Boston, MA etc.: Academic
  Press, Inc., 1988 (English).

\bibitem[{Dix}77]{Dix77}
J.~{Dixmier}, \emph{{\(C^*\)-algebras. Translated by Francis Jellett}},
  vol.~15, Elsevier (North-Holland), Amsterdam, 1977 (English).

\bibitem[{Eng}89]{Eng89}
R.~{Engelking}, \emph{{General topology.}}, vol.~6, Berlin: Heldermann Verlag,
  1989 (English).

\bibitem[FH20]{FaH20}
P.~K. {Friz} and M.~{Hairer}, \emph{{A course on rough paths. With an
  introduction to regularity structures. 2nd updated edition}}, 2nd updated
  edition ed., Cham: Springer, 2020 (English).

\bibitem[Foi02]{F2002}
L.~Foissy, \emph{Finite-dimensional comodules over the {H}opf algebra of rooted
  trees}, \href{http://dx.doi.org/10.1016/S0021-8693(02)00110-2}{J. Algebra
  \textbf{255} (2002)}, no.~1, 89--120.

\bibitem[FV10]{FV2010}
P.~K. Friz and N.~B. Victoir,
  \href{http://dx.doi.org/10.1017/CBO9780511845079}{\emph{Multidimensional
  stochastic processes as rough paths}}, Cambridge Studies in Advanced
  Mathematics, vol. 120, Cambridge University Press, Cambridge, 2010, Theory
  and applications.

\bibitem[FZ18]{FaZ18}
P.~K. {Friz} and H.~{Zhang}, \emph{{Differential equations driven by rough
  paths with jumps}}, {J. Differ. Equations} \textbf{264} (2018), no.~10,
  6226--6301 (English).

\bibitem[GIP15]{GIP15}
M.~Gubinelli, P.~Imkeller, and N.~Perkowski, \emph{Paracontrolled distributions
  and singular {PDE}s}, \href{http://dx.doi.org/10.1017/fmp.2015.2}{Forum Math.
  Pi \textbf{3} (2015)}, e6, 75.

\bibitem[GJR18]{GJR2018}
J.~Gatheral, T.~Jaisson, and M.~Rosenbaum, \emph{Volatility is rough},
  \href{http://dx.doi.org/10.1080/14697688.2017.1393551}{Quant. Finance
  \textbf{18} (2018)}, no.~6, 933--949.

\bibitem[Gub04]{Gub04}
M.~Gubinelli, \emph{Controlling rough paths},
  \href{http://dx.doi.org/10.1016/j.jfa.2004.01.002}{J. Funct. Anal.
  \textbf{216} (2004)}, no.~1, 86--140 (English).

\bibitem[Gub10]{Gub10}
M.~Gubinelli, \emph{Ramification of rough paths},
  \href{http://dx.doi.org/10.1016/j.jde.2009.11.015}{J. Differential Equations
  \textbf{248} (2010)}, no.~4, 693--721.

\bibitem[GVR21]{GVR21}
M.~Ghani~Varzaneh and S.~Riedel, \emph{Oseledets splitting and invariant
  manifolds on fields of {B}anach spaces},
  \href{http://dx.doi.org/https://doi.org/10.1007/s10884-021-09969-1}{J. Dyn.
  Diff. Equat. (2021)}.

\bibitem[GVRS22]{GRS22}
M.~Ghani~Varzaneh, S.~Riedel, and M.~Scheutzow, \emph{A dynamical theory for
  singular stochastic delay differential equations i: Linear equations and a
  multiplicative ergodic theorem on fields of banach spaces},
  \href{http://dx.doi.org/10.1137/21M1433435}{SIAM Journal on Applied Dynamical
  Systems \textbf{21} (2022)}, no.~1, 542--587.

\bibitem[Hai14]{Hai14}
M.~Hairer, \emph{A theory of regularity structures},
  \href{http://dx.doi.org/10.1007/s00222-014-0505-4}{Invent. Math. \textbf{198}
  (2014)}, no.~2, 269--504.

\bibitem[HK15]{HK2015}
M.~Hairer and D.~Kelly, \emph{Geometric versus non-geometric rough paths}, Ann.
  Inst. Henri Poincar{\'e} Probab. Stat. \textbf{51} (2015), no.~1, 207--251.

\bibitem[Hof03]{H2003}
M.~E. Hoffman, \emph{Combinatorics of rooted trees and {H}opf algebras},
  \href{http://dx.doi.org/10.1090/S0002-9947-03-03317-8}{Trans. Amer. Math.
  Soc. \textbf{355} (2003)}, no.~9, 3795--3811.

\bibitem[{Ina}19]{Ina19}
Y.~{Inahama}, \emph{{Rough path theory and stochastic calculus}}, {Sugaku
  Expo.} \textbf{32} (2019), no.~1, 113--136 (English).

\bibitem[{Laz}18]{Laz18}
A.~J. {Lazar}, \emph{{A selection theorem for Banach bundles and
  applications}}, \href{http://dx.doi.org/10.1016/j.jmaa.2018.02.008}{{J. Math.
  Anal. Appl.} \textbf{462} (2018)}, no.~1, 448--470 (English).

\bibitem[LM05]{LM2005}
T.~{Lada} and M.~{Markl}, \emph{{Symmetric brace algebras}},
  \href{http://dx.doi.org/10.1007/s10485-005-0911-2}{{Appl. Categ. Struct.}
  \textbf{13} (2005)}, no.~4, 351--370 (English).

\bibitem[LT19]{LT2019}
Y.~Liu and S.~Tindel, \emph{First-order {E}uler scheme for {SDE}s driven by
  fractional {B}rownian motions: the rough case},
  \href{http://dx.doi.org/10.1214/17-AAP1374}{Ann. Appl. Probab. \textbf{29}
  (2019)}, no.~2, 758--826.

\bibitem[LV07]{LaV07}
T.~{Lyons} and N.~{Victoir}, \emph{{An extension theorem to rough paths}},
  \href{http://dx.doi.org/10.1016/j.anihpc.2006.07.004}{{Ann. Inst. Henri
  Poincar\'e, Anal. Non Lin\'eaire} \textbf{24} (2007)}, no.~5, 835--847
  (English).

\bibitem[{Lyo}98]{Lyo98}
T.~J. {Lyons}, \emph{Differential equations driven by rough signals.},
  \href{http://dx.doi.org/10.4171/rmi/240}{Rev. Mat. Iberoamericana \textbf{14}
  (1998)}, no.~2, 215--310 (English).

\bibitem[{Man}08]{manHopf}
D.~{Manchon}, \emph{{Hopf algebras in renormalisation}}, {Handbook of algebra.
  Volume 5}, Amsterdam: Elsevier/Noth-Holland, 2008, pp.~365--427 (English).

\bibitem[MW21]{MW2021}
J.~Ma and H.~Wu, \emph{A fast algorithm for simulation of rough volatility
  models}, \href{http://dx.doi.org/10.1080/14697688.2021.1970213}{Quantitative
  Finance \textbf{0} (2021)}, no.~0, 1--16.

\bibitem[OG08]{OG2008}
J.-M. Oudom and D.~Guin, \emph{On the {L}ie enveloping algebra of a pre-{L}ie
  algebra}, \href{http://dx.doi.org/10.1017/is008001011jkt037}{J. K-Theory
  \textbf{2} (2008)}, no.~1, 147--167.

\bibitem[{Wea}18]{Weaver18}
N.~{Weaver}, \href{http://dx.doi.org/10.1142/9911}{\emph{{Lipschitz
  algebras}}}, Hackensack, NJ: World Scientific, 2018 (English).

\end{thebibliography}
\end{document}